\documentclass[10pt]{amsart}

\usepackage{graphicx, color}
\usepackage{epic}
\usepackage{eepic}

\newtheorem{thm}{Theorem}[section]
\newtheorem{ass}[thm]{Assumption}
\newtheorem{coro}[thm]{Corollary}
\newtheorem{lem}[thm]{Lemma}
\newtheorem{prop}[thm]{Proposition}

\theoremstyle{definition}

\newtheorem{defn}[thm]{Definition}

\newtheorem{nota}[thm]{Notation}

\theoremstyle{remark}

\newtheorem{remk}[thm]{Remark}

\newcommand{\underq}{\underline{q}}

\newcommand{\lk}{{\rm Lk}}

\newcommand{\Gal}{{\rm Gal}}

\newcommand{\Card}{{\rm Card}}

\newcommand{\R}{{\mathbb{R}}}
\newcommand{\Z}{{\mathbb{Z}}}
\newcommand{\N}{{\mathbb{N}}}

\newcommand{\Scal}{ {\mathcal S}}

\newcommand{\Ecal}{ {\mathcal E}}

\newcommand{\dev}{{\bf dev}}

\begin{document}

\title{Ellipses in translation surfaces}

\author[S. A. Broughton]{S. Allen Broughton}

\address{Department of Mathematics \\  
Rose-Hulman Institute of Technology  \\
5500 Wabash Ave.  \\
Terre Haute, IN 47803 USA}  

\email{allen.broughton@rose-hulman.edu}

\author[C. M. Judge]{Chris Judge}

\address{Department of Mathematics \\
Indiana University \\
Bloomington, IN  47405 USA}

\email{cjudge@indiana.edu}

\thanks{The first author thanks Indiana University for its hospitality. 
       The second author thanks the Max-Planck Institut f\"ur Mathematik (Bonn),
       the Institut Fourier, and the \'Ecole Polytechnique F\'ed\'eral de Lausanne
       for their hospitality and support.}

\begin{abstract}
We study the topology and geometry of the 
moduli space of immersions of ellipses into a translation surface. 
The frontier of this space is naturally stratified by the number of 
`cone points' that an ellipse meets. The stratum consisting of ellipses
that meet three cone points is naturally a two dimensional (non-manifold) polygonal 
cell complex.  We show that the topology of this cell-complex 
together with the eccentricity and direction of each of its vertices
determines the translation surface up to homothety. As a corollary we characterize 
the Veech group of the translation surface in terms of automorphisms 
of this polygonal cell complex.   
\end{abstract}

\maketitle

\section{Introduction}

A translation structure $\mu$ on a (connected) topological surface $X$ is an equivalence
class of atlases whose transition functions are translations. 
Translation surfaces are fundamental objects in Teichm\"uller theory, the study of
polygonal billiards,  and the study of interval exchange maps.

Cylinders that are isometrically embedded in a translation surface 
play a central role in the theory. 
In Teichm\"uller theory, they appear as solutions to moduli problems \cite{Strebel}.
In rational billiards and interval exchange maps, cylinders correspond to periodic orbits
\cite{HMSZ}  \cite{MasurTabachnikov} \cite{Smillie}.

Each periodic geodesic $\gamma$ on a translation surface belongs
to a unique `maximal' cylinder that is foliated by 
the geodesics that are both parallel and homotopic to $\gamma$. 
One method for producing such periodic geodesics
implicitly uses ellipses. Indeed, if $X$ admits an isometric 
immersion of an ellipse with area greater than that of $X$,  
then the image of the immersion contains a cylinder, 
and hence a periodic geodesic \cite{MasurSmillie} \cite{Smillie}. 

Ellipses interiors also serve to interpolate between maximal cylinders.
The set, $\Ecal(X, \mu)$, of ellipse interiors isometrically immersed in $X$ 
has a natural geometry coming from the space of quadratic forms
on $\R^3$ (see \S \ref{SectionSubconicsTranslation}).  The set of maximal cylinders is a 
discrete set lying in the frontier of the path connected space $\Ecal(X, \mu)$.  

If the frontier of a translation surface $X$ is finite, then each point in the 
frontier may be naturally regarded as a {\em cone point} with angle 
equal to an integral multiple of $2 \pi$.
If the frontier of an immersed ellipse interior $U$ contains a cone point $c$,
then we will say that $U$ {\em meets} $x$.  
If an ellipse interior meets a cone point, then the ellipse interior 
belongs to the frontier of $\Ecal(X, \mu)$. 
The remainder of the frontier consists of cylinders. 

The number of cone points met by an ellipse interior determines 
a natural stratification of $\Ecal(X,\mu)$.\footnote{To be 
precise one must lift to the universal cover before counting. See \S
\ref{SectionSubconicsTranslation}.}
In  \S \ref{HomotopySection} we demonstrate that $\Ecal(X, \mu)$  
is homotopy equivalent to the 
stratum consisting of ellipse interiors that meet at least three cone points. In 
\S \ref{SectionCellComplex}, we prove that the completion of this stratum is naturally 
a (non-manifold) 2-dimensional cell complex whose 2-cells are convex polygons. 

We show that the topology of this polygonal complex and the geometry 
of the immersed ellipses and cylinders that serve as its vertices  
together encode the geometry of $(X, \mu)$ up to homothety.  

\begin{thm} \label{Rebuild}
Suppose that there is a homeomorphism $\Phi$ that maps the 
polygonal complex associated to $(X, \mu)$ onto the polygonal
complex associated to $(X', \mu')$.
If for each vertex $U$, the ellipses (or strips) 
$U$ and $\Phi(U)$ differ by a homothety, 
then $(X,\mu)$ and $(X', \mu')$ are equivalent up to homothety. 
\end{thm}

Affine mappings naturally act on planar ellipses,
and hence the group of (orientation preserving) 
affine homeomorphisms of $(X,\mu)$ 
acts on $\Ecal(X,\mu)$. Because $\mu$ is a translation structure, 
the differential of an affine homeomorphism is a well-defined
$2 \times 2$ matrix of determinant 1  \cite{Veech89}. The set of
all differentials is a discrete subgroup of $SL_2(\R)$
that is sometimes called the {\em Veech group} and is denoted $\Gamma(X, \mu)$.
Using Theorem \ref{Rebuild}, one can characterize $\Gamma(X, \mu)$.

\begin{thm} \label{Auto}
The group $\Gamma(X, \mu)$ consists of the $g \in SL_2(\R)$ for which
there exists an orientation preserving
self homeomorphism of the polygonal complex associated to $(X, \mu)$
such that for each vertex $U$ 
there exist a homothety $h_U$ such that $U$ differs
from $\Phi(U)$ by  $h_U \circ g$. 
\end{thm}

The group $\Gamma(X, \mu)$ is closely related to the subgroup of the 
mapping class group of $X$ that stabilizes the 
Teichm\"uller disc associated to $(X,\mu)$ \cite{Veech89}. To be precise, 
each mapping class in the stabilizer has a unique representative
that is affine with respect to $\mu$. The Veech group is the 
set of differentials of these affine maps, and is isomorphic 
to the stabilizer modulo automorphisms. In particular, 
if there are no nontrivial automorphisms in the stabilizer, 
then the quotient of the hyperbolic plane by a lattice 
Veech group is isometric to a Teichm\"uller curve. 

There is a natural map that sends each ellipse interior 
$U \subset \R^2$ to the coset of $SO(2) \setminus SL_2(\R)$ consisting of $g$
such that $g(U)$ is a disc. This map naturally determines a map 
from $\Ecal(X, \mu)$ onto the Poincar\'e disc. The image of the 1-skeleton of the 
polygonal cell complex determines a tessellation of the upper half-plane 
that coincides with a tessellation defined by William Veech \cite{Veech99} \cite{Veech08} 
and independently Joshua Bowman \cite{Bowman08}.  Indeed, our work 
began with a reading of a preprint of \cite{Veech08}.
In a companion paper \cite{BrtJdg09}, we will discuss these
connections in more detail.

To develop the deformation theory of immersed ellipses, 
we use quadratic forms on $\R^3$. Each ellipse interior in an affine plane $P \subset \R^3$  
is a sublevel set, $\left\{\vec{x} \in P~ |~ q(\vec{x})< 0 \right\}$,
of a quadratic form $q$. The space of quadratic forms on $\R^3$ is a
six dimensional real vector space, and two quadratic forms determine 
the same ellipse interior if and only if they differ by a 
positive scalar.  

A cylinder is the image of an isometric immersion of a {\em strip}, 
the interior of the convex hull of two parallel lines in $\R^2$.
A strip is also a sublevel set of a 
quadratic form $q$ restricted to an affine plane $P$.

The language of quadratic forms unifies the treatment of immersed ellipses 
and cylinders. We define a planar {\em subconic} to be a sublevel set  
$\{(x,y)~ |~ q(x,y,1)<0\}$ where $q$ is a quadratic form on $\R^3$.
The set of planar subconics includes ellipse interiors and strips, 
but also includes ellipse exteriors, parabola interiors, etc. 

We will let $\Scal(X, \mu)$ denote the space of immersions 
of subconics into a translation surface $(X,\mu)$.
We let $\Scal_n(X, \mu)$ denote the stratum consisting of 
subconics that meet at least $n$ cone points. 
  
The translation structure $\mu$ determines a canonical Euclidean metric
on the surface $X$. The completion of $X$ with respect to this metric will be denoted
by $\overline{X}$.  If $A$ is a subset of $X$, then $\partial A$ will denote the 
complement of the interior of $A$ in its closure in $\overline{X}$.  In particular, 
$\partial X$ we will denote the set of cone points. 

Other notation can be found in the table that is located 
at the end of this introduction.

\vspace{.5cm}

\paragraph{{\bf Outline of paper}} In sections \S \ref{PreliminariesSection}
 through  \S \ref{SectionZ} we establish notation and introduce basic tools.
In \S \ref{PreliminariesSection}
we recall the basic theory of translation surfaces including the 
developing map, and in \S \ref{SectionQuadratic} we recall some basic
facts about quadratic forms. In \S \ref{SectionPlanarSubconics}, we 
collect elementary facts about the subconics in the plane.
In \S \ref{SectionZ} we consider subconics in the plane determined by finite
sets of points.

In \S \ref{SectionSubconicsTranslation} we define the geometry of the 
space of immersed subconics in a translation surface. We show, for example, that 
$\Scal(X, \mu)$ is naturally a 5-dimensional real-projective manifold. 

Beginning in \S \ref{SectionCovering} we restrict attention to 
the translation surfaces that cover a precompact translation surface. 
For such surfaces, the only possible
subconics are ellipse interiors and (immersed) strips.  
We show that $\Scal_5(X, \mu)$ is discrete and that only finitely
many members of $\Scal_5(X, \mu)$ contain a given nonempty open subset of $X$.

In \S \ref{HomotopySection} we show that the subspace $\Ecal_3(X) \subset \Scal_3(X)$
consisting of ellipse interiors in homotopy equivalent to $X$. 
As a consequence, $\Ecal_3(\tilde{X}, \tilde{\mu})$ is the universal cover of 
$\Ecal_3(X, \mu)$.  This fact is used crucially in our proof of 
Theorem \ref{Rebuild}.

In \S \ref{SectionCellComplex} we define the cell
structure on $\Scal_3(X, \mu)$. In particular, each 2-cell $\Scal_Z$
corresponds to a triple $Z \subset \partial X$ that defines a triangle in 
$\overline{X}$.  We show that $\overline{\Scal}_Z$ may be regarded 
as the convex hull of all subconics that lie in 
$\overline{\Scal}_Z  \cap \Scal_5(X, \mu)$. In \S \ref{SectionRealizable},
we characterize those triples (resp. quadruples) in $\partial X$ 
that determine a 2-cell (resp. 1-cell) in $\Scal_3(X, \mu)$.
 
The set $\overline{\Scal}_Z  \cap \Scal_5(X, \mu)$ determines
an orientation of each 2-cell. In \S \ref{SectionOrientation},  
we relate this orientation to the ordering of $\partial U \cap \partial X$ 
where $\overline{\Scal}_Z  \cap \Scal_5(X, \mu)$.
In \S  \ref{SectionLink} we study the oriented link of each vertex
in the 2-skeleton. For example, we show that the link of $U$ is determined
up to isomorphism by the cardinality of $\partial U \cap \partial X$. 
In \S \ref{SectionFrontier}, we use the analysis of the 
oriented links to show that a homeomorphism 
$\Phi: \Scal(X, \mu) \rightarrow \Scal(X', \mu')$ determines 
a unique bijection $\beta: \partial X \rightarrow \partial X'$
such that for each $Z \subset \partial U \cap \partial X$,
we have $\Phi(\Scal_Z)=\Scal_{\beta_Z}$.  Note that the results 
of these three sections are purely combinatorial. 

In \S \ref{SectionBisectors} we prove the main geometric lemma.
Roughly speaking, the sides of a cyclic polygon are determined up
to translation by the directions of the  bisectors of successive
diagonals in the polygon. In \S \ref{SectionRebuild}, we combine
this geometric Lemma with the combinatorial results to prove
Theorem \ref{Rebuild}.  In \S \ref{SectionVeech},  we prove Theorem 
\ref{Auto}.

\vspace{.5cm}
 
\paragraph{{\bf Acknowledgments}}  We thank Matthew Bainbridge, Alex Eskin,
Fran\c{c}ois Gu\'eritaud, Martin M\"oller, and especially Saul Schleimer for discussions
and encouragement.  We also thank the referree for very helpful comments.

\vspace{.5cm}
 
\paragraph{{\bf Notation}}
For the convenience of the reader, we include the following list of notation.

\vspace{.4cm}

\begin{tabular}{lll}
  $X$  & \S \ref{PreliminariesSection} & topological surface    \\

  $\tilde{X}$ & \S \ref{PreliminariesSection} &  universal cover of $X$    \\

  $\Gal(\tilde{X}/X)$ & \S \ref{SectionSubconicsTranslation} 
  & group of covering transformations   \\

  $\mu$ &  \S \ref{PreliminariesSection}  &  translation atlas/structure  \\

 $\dev_{\mu}$ & \S \ref{PreliminariesSection} & developing map associated to $\mu$  \\

 $\overline{X}$ & \S \ref{PreliminariesSection} & metric completion of $X$   \\

 $\partial X$ & \S \ref{PreliminariesSection} & frontier of $X$  \\

 $\Card(F)$ & \S \ref{SectionQuadratic} & cardinality of a set $F$ \\

 $Q(\R^n)$ & \S \ref{SectionQuadratic} & vector space of real quadratic forms on $\R^n$  \\

 $Q_F$ & \S \ref{SectionQuadratic} &  quadratic forms that vanish on $F$     \\

 $\underline{q}$ &    \S \ref{SectionPlanarSubconics}
& restriction of quadratic form $q$ to first two coordinates
                        \\

 $\widehat{(x_1,x_2)}$  &  \S \ref{SectionPlanarSubconics} &   $(x_1,x_2,1)$ \\ 

 $U_q$  & \S \ref{SectionPlanarSubconics}  &  subconic, set of $x$ such that $q(\hat{x})<0$  \\

 $\partial U_q$  & \S \ref{SectionPlanarSubconics} & boundary of subconic   \\

 $\Scal_Z$ & \S \ref{SectionZ}
   & set of subconics $U$ such that $Z \subset \partial U \cap \partial X$ \\

 $\Ecal_Z$ & \S \ref{SectionZ}
 & set of ellipse interiors $U$ such that $Z \subset \partial U \cap \partial X$ \\
 
 $\Scal(X, \mu)$  &  \S \ref{SectionSubconicsTranslation}
     & set of subconics in a translation surface \\ 

 $\Scal_n(X, \mu)$   &  \S \ref{SectionSubconicsTranslation}
         &  stratum of  $\Scal(X, \mu)$ \\ 

 $\Ecal(X, \mu)$  &  \S \ref{HomotopySection}
    & set of ellipse interiors in a translation surface \\

 $\Ecal_n(X, \mu)$ & \S \ref{HomotopySection}  & stratum of $\Ecal(X, \mu)$ \\ 

 $s_U$ & \S \ref{SectionOrientation}   & successor function of $\partial U \cap \partial X$ \\

$Z_U(x,x')$ & \S \ref{SectionOrientation}   & $\{x, s_U(x), x', s_{U}(x')\}$ \\

${\rm Lk}(U)$ & \S \ref{SectionLink}  & link of the vertex $U$ in $\Scal_3(X, \mu)$ \\

$\Phi$  & \S \ref{SectionFrontier} &  isomorphism of polygonal cell complexes \\

$\beta_U$  & \S \ref{SectionFrontier} & map on $\partial U \cap \partial X$ \\

$[U]$ & \S \ref{SectionBisectors} 
   & orbit of $U$ under the group generated  \\

& & by homotheties and translations. 

\end{tabular}

\vspace{.4cm}


\section{Preliminaries on translation surfaces} \label{PreliminariesSection}

Recall that a {\em translation atlas} on a surface $X$ is an
atlas $\{\mu_{\alpha}:U_{\alpha} \rightarrow \R^2 \}$ such that 
each transition map $\mu_{\beta} \circ \mu_{\alpha}^{-1}$ is a translation. 
A translation structure $\mu$ on $X$ is an equivalence class of translation
atlases where two atlases are equivalent if the they agree on their
common refinement up to translations. 

Let $(X, \mu)$ and $(Y, \nu)$ be translation surfaces. 
A mapping $f: X \rightarrow Y$ is called a translation map with respect to 
$\mu$ and $\nu$---or simply a {\em translation}---iff for each chart 
$\mu_{\alpha}$ in $\mu$ and  $\nu_{\beta}$
in $\nu$, the map $\nu_{\beta} \circ f \circ \mu_{\alpha}^{-1}$ is a translation.

One may define a metric structure on $(X, \mu)$ in several
equivalent ways.  One approach is to simply pull back 
the Riemannian metric tensor $dx^2+ dy^2$ on 
the Euclidean plane via  a translation atlas. 
Since translations preserve $dx^2+ dy^2$, the resulting metric
tensor on $X$ is well-defined. Thus, $\mu$ determines  
a (Riemannian) distance function $d_{\mu}:X \times X \rightarrow \R$.

Let $\overline{X}$ denote the metric completion of $X$ with 
respect to $d_{\mu}$. 

\begin{defn} \label{DefnFrontier}
The {\em frontier} of $X$ with respect to $\mu$ is the set 
$\partial X:=\overline{X} \setminus X$.
\end{defn}

A translation structure pulls-back under a covering map.
In particular, one associates to each translation structure $\mu$ on $X$
a translation structure $\tilde{\mu}$ on the universal 
cover $\tilde{X}$. 

The following assumption will be implicit throughout this paper.

\begin{ass} \label{DiscreteAss}
The frontier $\partial X$ 
is a discrete subset of  $\overline{X}$.
Moreover, the frontier $\partial \tilde{X}$
of the universal cover $\tilde{X}$ 
is a discrete subset of the completion of 
$\tilde{X}$. 
\end{ass}

Assumption \ref{DiscreteAss} implies that the frontier 
of a precompact\footnote{Here we use `precompact' 
as a synonym for totally bounded so that completing the space gives
a compact space.} translation surface 
is finite. In addition, a precompact translation surface 
has finite genus.

A translation structure is an example of a geometric structure in the sense
of \cite{ThrBook}. If $X$ is simply connected, for each
translation structure $\mu$ on $X$, there exists a map
$\dev_{\mu}: X \rightarrow \R^2$ such that the set of restrictions 
to convex sets is an atlas 
that represents $\mu$. This {\em developing map} is unique up to
post-composition by a translation. 
Note that  $\dev_{\mu}$ extends uniquely
to a continuous function on $\overline{X}$. 
We will abusively use $\dev_{\mu}$ to denote this extension.

In general, the developing map is not injective. 
But the restriction of $\dev_{\mu}$ to a `star convex' subset 
is injective. To be precise,  a set $Y \subset X$ 
will be said to be {\em star convex} with respect 
to $x$ if and only if for each $x' \in Y$, there exists a geodesic
segment in $Y$ joining $x'$ to $x$. A subset $Y \subset X$ will be called convex 
if and only if it is star convex with respect to each
of its points. 

It will be convenient to extend the notions
of convexity and star convexity  
to subsets of the completion $\overline{X}$. 
We will say that $Y \subset \overline{X}$ 
is star convex with respect to $y$ if and only if
for each $y' \in Y$, there exists a convex subset $C \subset X$
such that $y, y' \in \overline{C}$.  A set $Y \subset \overline{X}$ 
will be called  convex if and only if it is star convex 
with respect to each $y \in Y$.

\begin{remk}  \label{ConvexityRemark}
Our notions of convexity for subset of $\overline{X}$
are defined in a way that 
assures that the restriction of the developing map is injective.  
Note that the definition differs significantly from the
one that is natural if $\overline{X}$ is regarded
as a length space.  If  $\partial X$ is discrete,
then the geodesics in the length space $\overline{X}$
are finite concatenations of completions of geodesic segments in $X$.
From the perspective of length spaces, each geodesic is (star) convex,
but from our viewpoint many geodesics are not (star) convex. 
\end{remk}

Let $(X, \mu)$ be a simply connected translation surface.
Given $A  \subset X$, define the {\em inradius of $A$} with respect to $\mu$
to be the supremum of the radii of disks that can be embedded into $A$.

\begin{prop} \label{Inradius}
If $(X,\mu)$ is the universal cover of a precompact translation surface 
with nonempty discrete frontier, then the inradius of $X$ is finite. 
\end{prop}

\begin{proof}
By assumption we have a precompact translation surface $(Y, \nu)$
and a covering $p: X \rightarrow Y$ such that $p^*(\nu)=\mu$. 
Let $d_X: \overline{X} \times \overline{X} \rightarrow \R$ and 
$d_Y: \overline{Y} \times \overline{Y} \rightarrow Y$ denote the respective distance
functions on the respective completions of $X$ and $Y$. 
Note that $p$ is a local isometry with respect to these distance functions. 

The function $y \mapsto d_Y(y,\partial Y)$ is continuous.
Thus since $\overline{Y}$ is compact,  there exists $K_Y$ such 
that for all $y \in Y$, we have $d_Y(y,\partial Y)<K_Y$. 

For each $x \in X$, we have $d_X(x, \partial X) = d_Y(p(x), \partial Y)$.
Indeed, since $\partial Y$ is finite, there exists a geodesic
segment $\gamma$  joining $p(x)$ to $\partial Y$ whose length equals
$d_Y(p(x), \partial Y)$.  This geodesic segment lifts under $p$ to a unique 
path $\tilde{\gamma}$ that joins $x$ to $\partial X$. Since one endpoint
of $\tilde{\gamma}$ lies in $\partial X$ and $(X, \mu)$ is a translation
surface, the restriction of $p$ to ${\tilde{\gamma}}$ is an isometry
onto $\gamma$. In particular, the length of $\tilde{\gamma}$ equals
the length of $\gamma$. Thus,  $d_X(x, \partial X) = d_Y(p(x), \partial Y)$. 

Therefore, for each $x \in X$, we have $d_X(x, \partial X) < K_Y$. 
It follows that if $D$ is a Euclidean disk embedded in $X$, then 
the radius of $D$ is less than $K_Y$. 
\end{proof}


\section{Quadratic forms on $\R^3$}   \label{SectionQuadratic}

We recall some basic facts about quadratic forms 
defined over the real numbers. See chapter 14 in \cite{Berger}.

Let $Q(V)$ denote the vector space of all real-valued quadratic forms on 
a real vector space $V$.
Given $q \in Q(V)$ we let $q(\cdot, \cdot)$ denote its polarization.
There is a basis for $V$ such that 
the matrix associated to the polarization 
is diagonal with entries belonging to $\{+1,-1,0\}$ (Sylvester's law).
Let $n_+(q), n_-(q)$, and $n_0(q)$ denote, respectively, the number of 
diagonal entries that are $+1$, $-1$, and $0$.  
We say that $q$ has {\em signature} $(n_+(q), n_-(q))$.
If $n_0(q) \neq 0$, then we say that $q$ is {\em degenerate}.

The {\em radical}, ${\rm rad}(q)$, of (the polarization of) the quadratic form $q$ is the 
set of $v$ such that $q(v,w)=0$ for all $w \in \R^3$.  The form $q$ is nondegenerate 
iff only ${\rm rad} (q)=\{0\}$.  If ${\rm rad}(q) \neq \{0\}$, then $0$ is an eigenvalue
and ${\rm rad}(q)$ is the associated eigenspace.

We will be interested in quadratic forms that vanish on a prescribed
subset $F \subset \R^3$.  Define 
\[  Q_F~ =~ \left\{  q \in Q(\R^3)~ |~  q(F)~ =~ \{0\} \right\}. \]
Note that $Q_F$ is a vector subspace.

To each point $v \in \R^3 \setminus \{0\}$ we may associate the unique 
line $\ell(v)$ containing $v$ and the origin.  
We say that a set of points $F \subset \R^3$ is {\em in general position} if 
no three of the lines in $\ell(F)$ are coplanar.

To construct quadratic forms that vanish on subsets of $F \subset \R^3$, 
one can use the collection, $(\R^3)^*$ of linear forms on $\R^3$. 
Given linearly independent vectors $v, w \in \R^3$,
let $L_{vw} \subset (\R^3)^*$ be the collection of linear forms
that vanish on the plane spanned by $v$ and $w$. Note that a product $\eta \cdot \eta'$ 
of linear forms $\eta \in L_{xy}$ and $\eta' \in L_{zw}$ is a quadratic form 
in $Q_{\{x,y,z,w\}}$. 
 
Let $\Card(F)$ denote the cardinality of a set $F$.

\begin{prop} \label{Dimension}
If $\Card(F) \leq 5$ and $F$ is general position,  then $\dim(Q_F)=6- \Card(F)$.   
\end{prop}

\begin{proof}
Let $F=\{x_1, \ldots, x_n\}$ where $n= \Card(F)$. 
Note that for each $i$,  the map $q \mapsto q(v_i)$ is a linear 
functional on $Q(\R^3)$.  The set $Q_F$ is the kernel of the homomorphism
$\phi: Q(\R^3) \rightarrow \R^n$ defined by
\[ \phi(q)~ =~  \left( q(v_1), \ldots, q(v_n) \right). \]
Since $\dim(Q(\R^3))=6$, it suffices to show that $\phi$ 
has full rank. If $n=1$, then this is true.  If $1<n\leq 5$, then
suppose that there exist $(a_1, \ldots, a_n) \in \R^n$ such that 
for each $q \in Q(\R^3)$ we have $\sum_i a_i \cdot q(v_i)=0$. 

Since $F$ is in general position, for each $j$ there exists 
$q_j \in Q(\R^3)$ such that $q_j(v_j)\neq 0$ and $q_j(v_i)= 0$ if $i\neq j$.
For example, if $n=5$ and $j=j_1$, then choose $\eta \in L_{v_{j_2},v_{j_3}}$,
$\eta' \in L_{v_{j_{4}}v_{j_5}}$, and set $q= \eta \cdot \eta'$.

Hence $a_j \cdot q_j(v_j) =0$ and therefore $a_j=0$ for each $j$.
Thus, $\phi$ has full rank.
\end{proof}

We also have the following variant of the preceding proposition
that will be used in \S \ref{SectionRealizable}.
Let $dq_x: \R^3 \rightarrow \R$ denote the differential of $q$ at $x$.

\begin{prop} \label{Dimension2}
Let $n =3$ or $4$, and let $\{v_1, \ldots, v_n\} \subset \R^3 \setminus \{0\}$ 
be in general position. Let $\{w_{1} \ldots w_{5-n}\}$ be a subset 
of $\R^3 \setminus \{0\}$ such that for each $i,k$, the vector
$w_k$ does not belong to the plane spanned by $v_k$ and $v_i$.
Then the set of quadratic forms $q$ with $q(v_i)= 0$ and 
$dq_{v_i}(w_i)=0$ is a 1-dimensional vector space.   
\end{prop}

\begin{proof}
The set in question is the kernel of the homomorphism
$\phi: Q(\R^3) \rightarrow \R^5$ defined by
\[ \phi(q)~ =~  \left( q(v_1), \ldots, q(v_n), dq_{v_1}(w_1), \cdots, dq_{v_{5-n}}(w_{5-n}) 
   \right). \]
It suffices to show that $\phi$ has full rank.

Suppose that there exist $(a_1, \ldots, a_n) \in \R^n$ and $(b_1,\ldots, b_{5-n})$
such that for each $q \in Q(\R^3)$ we have 
\[  \sum_i a_i \cdot q(v_i)~ +~ \sum_k b_k \cdot dq_{v_k}(w_k)~ =~ 0.  \] 
Since $F$ is in general position and each $w_k$ does not belong to 
the plane spanned by $v_k$ and $v_i$, for each $i=1,\ldots, n$, 
there exists $q_i \in Q(\R^3)$ such that $q_i(v_i) \neq 0$ but 
$q_i(v_j)=0$ for each $j \neq i$ and $d(q_i)_{v_k}(w_k)=0$ for each $k$.
Similarly, for each $k=1, \ldots, 5-n$, there exists $q_k'$ so that 
$d(q_k')_{v_k}(w_k) \neq 0$ but 
$q_i(v_j)=0$ for each $j$ and $dq_{v_j}(w_j)=0$ for each $j \neq k$.

It follows that $a_i=0$ for each $i$ and $b_k=0$ for each $k$. 
Thus, $\phi$ has full rank.
\end{proof}

Let $F \subset \R^3$ be a triple. For each $x \in F$, define $D_{x} \subset Q_F$ 
to be the set of quadratic forms $d$ such that $d(x,v)=0$ 
for all $v \in \R^3$. Note that the set $D_x$ is a 1-dimensional subspace. 
If $d \in D_x$ and $d \neq 0$, then  
\[ d^{-1}\{0\}= \bigcup_{y \in F \setminus \{x\}} \langle x, y \rangle. \]
where $\langle x, y \rangle$ is the plane spanned by $x$ and $y$. 

\begin{lem} \label{Triple}
Let $F=\{v_1,v_2,v_3 \}$ and let $d_i \in D_{v_i}$.
If the triple $F$ is in general position, 
then $\{d_1,d_2,d_3\}$ is a basis for $Q_F$.  
\end{lem}

We will call  $\{d_1, d_2, d_3\}$ a {\em natural basis} for $Q_F$.

\begin{proof}
Suppose that there exist $a_i \in \R$ such that 
$\sum a_i \cdot d_i(v)=0$ for all $v \in \R^3$. Let $\{i,j,k\}$ be
distinct indices. Note that both $d_i$ and $d_j$ vanish identically on the
span of $\{v_{i}, v_j\}$, but since $\ell(F)$ is not coplanar,
the form $d_k$ does not. It follows that $a_{k}=0$.  
Hence the forms  $\{d_1,d_2,d_3\}$ are independent, 
and the claim follows from Proposition \ref{Dimension}.
\end{proof}

\begin{coro} \label{Type11}
If the triple $F$ is in general position and $q \in Q_F \setminus \{0\}$
is degenerate, then $q$ is of type $(1,1)$.
\end{coro}

If $F$ is a triple in general position, then the vector space $Q_F$ has a canonical 
orientation. Indeed, note that since $F$ is in general position, 
then $F$ has a canonical cyclic ordering: If $F= \{x,y,z\}$, then we say that 
$y$ follows $x$ if and only if  $(y-x, z-x, x)$ is an oriented basis for $\R^3$.
We say that a natural basis $\{q_x~ \in Q_x~ |~ x \in F\}$ is {\em negative}
if and only if for each $p$ in the interior of the convex hull of $F$
and each $x \in F$ we have $q_x(p)<0$. We say that an ordered negative natural basis  
$(d_{1},d_{2},d_{3})$ is {\em oriented} if and only if the corresponding 
ordered triple $(x_1,x_2,x_3) \in F^3$ is cyclically ordered. 
In general, a basis of $Q_F$ is oriented if and only if it differs from 
an oriented negative basis by a linear transformation having positive determinant.

We now turn to quadruples $F \subset \R^3$.

\begin{lem} \label{Quadruple}
If the quadruple $F$ is in general position, then 
the set of degenerate quadratic forms in $Q_F$ is a union
of three distinct lines. 
\end{lem}

\begin{proof}
Let $F=\{v_{1},v_2,v_3,v_4\}$. 
For $\{i,j,k, \ell\}=\{1,2,3,4\}$, define $Q_{ij,k\ell}$ to be the 
set of pairwise products of linear forms from $L_{v_iv_j}$ and $L_{v_jv_k}$. 
Each $Q_{ij, k\ell}$ is a line in $Q_F$, and there are three distinct
such lines. 

Let $q \in Q_F$ be degenerate. It follows from Corollary 
\ref{Type11} that $q$ is of type $(1,1)$. Thus, there exist 
linear forms $\eta, \eta' \in (\R^3)^*$ such that $q=\eta \cdot \eta'$.
Since $q \in Q_F$, for each $i=1,2,3$, either  $\eta(v_i)=0$ or  $\eta'(v_i)=0$. 
Since $F$ is in general position, $\eta$ and $\eta'$ each vanish on 
exactly two points in $F$ and since $q \in Q_F$ these pairs are distinct. 
Thus, $q$ belongs to one of the distinct lines.
\end{proof}


\section{Subconics in the plane} 
                 \label{SectionPlanarSubconics}

Recall that a conic section is the intersection of a hyperplane in $\R^3$  
and  the zero level set of a quadratic form $q \in Q(\R^3)$.  
In the context of translation surfaces, 
we are mainly interested in the intersection of 
the sublevel set $q^{-1}\{(-\infty, 0)\}$ and a 
hyperplane. For simplicity, we choose this hyperplane to be 
$\{(x_1,x_2,1):~ (x_1,x_2) \in \R^2\}$. The following notation will 
be convenient.

\begin{nota}
Given a point $x=(x_1,x_2) \in \R^2$, let $\hat{x}=(x_1,x_2,1)$.
Given a subset $Z \subset \R^2$, let $\hat{Z}= \{ \hat{x}~ |~ x\in Z\}$. 
\end{nota}

\begin{defn}
The {\em subconic associated to $q$} is the set 
\[ U_{q}~ =~ \left\{ x \in \R^2~ |~ q \left(\hat{x} \right)~ <~ 0 \right\}. \] 
The subconic $U_q$ said to be {\em nontrivial} if and only if $U_q$ is
a proper subset of $\R^2$. The {\em boundary} of $U_q$ is the set 
\[ \partial U_q~ =~ \left\{ x \in \R^2~ |~ q \left( \hat{x} \right)~ =~ 0 \right\}. \] 
\end{defn} 

Given a quadratic form $q \in Q(\R^3)$, let $\underq$ denote the 
restriction of $q$ to the the plane  $\{ (x,y,0)~ |~ (x,y) \in \R^2\}$.
Subconics can be classified according to the signatures
of $q$ and $\underline{q}$.  

\begin{defn}
A subconic $U_q$ is called an {\em ellipse interior} if and only if $q$ has 
signature $(2,1)$ and $\underline{q}$ has signature $(2,0)$.
\end{defn}

\begin{defn}
A subconic $U_q$ is called a {\em strip} if and only if $q$ has 
signature $(1,1)$ and $\underline{q}$ has signature $(0,0)$.
\end{defn}

\begin{defn}
A subconic $U_q$ is called a {\em half-plane} if and only if $q$ has 
signature $(1,1)$ and $\underline{q}$ has signature $(0,0)$.
\end{defn}

\begin{defn}
A subconic $U_q$ is called a {\em parabola interior} if and only if $q$ has 
signature $(2,1)$ and $\underline{q}$ has signature $(1,0)$.
\end{defn}

\begin{remk}
In this paper, we will be almost exclusively 
concerned with ellipse interiors and strips.
See Proposition \ref{EllipseStrip}.
\end{remk}

The following propositions are elementary.

\begin{prop} \label{ConvexClassification}
A non-trivial subconic is convex if and only if 
it is an ellipse interior, a parabola interior, a strip, or a half-plane. 
\end{prop}

\begin{prop} \label{EllipseBounded}
A nontrivial subconic is bounded if and only if it is an ellipse interior. 
\end{prop}

Let $\Scal(\R^2)$ denote the set of all nontrivial subconics. 
The natural topology on $Q(\R^3)$ induces a topology on $\Scal(\R^2)$.
Namely, a collection  of subconics is said to be open
if and only if the set of corresponding quadratic forms is open.

The following propositions are elementary.

\begin{prop} \label{EllipseOpenFrontier}
The set of ellipse interiors is open in $\Scal(\R^2)$, and its
frontier equals the set of parabola interiors, strips, and half-planes.
\end{prop}

\begin{prop} \label{EllipseConvex}
The set of quadratic forms asociated to ellipse interiors 
is convex and is preserved by the $\R^+$ action.
\end{prop}


\section{Subconics determined by a finite set of points}

\label{SectionZ}

Given a finite set $Z \subset \R^2$, let 
$\Scal_Z$ denote the set of nontrivial subconics whose boundary contains $Z$,
and let  $Q_Z$ denote the set of quadratic forms $q$ such that 
$q(\hat{Z})= \{0\}$.

Note that $\Scal_Z  \subset \{ U_q~ |~ q \in Q_Z \}$.
The opposite inclusion does not hold in general. For example,
if $Z$ lies on a line, then let  $\eta$ be a nontrivial linear 
form that vanishes on a plane containing $\hat{Z}$.  
The quadratic form $\eta^2$ belongs to $Q_Z$ but
 $U_{\eta^2}$ is the empty set.  

\begin{lem} \label{MissingNothing}
Suppose that $Z \subset \R^2$ is not collinear.
If $U_q$ is empty for some $q \in Q_Z$, then $q \equiv 0$.
In particular, $\Scal_Z = \{ U_q~ |~ q \in Q_Z\setminus \{0\} \}$.  
\end{lem}

\begin{proof}
Suppose that $U_q = \emptyset$ for some $q \in Q_Z$.
Given a line $\ell \subset \R^2$, the restriction of $q$ to $\hat{\ell}$ 
is a nonnegative quadratic function. Thus, if $x$ and $y$ are distinct 
points in $\ell$ such that $q(\hat{x})=0=q(\hat{y})$, then 
$q$ vanishes identically on $\hat{\ell}$.  

Let $z_1$, $z_2$, and $z_3$ be three noncolinear points in $Z$. 
Since $q \in Q_Z$, the form $q$ vanishes on the lines determined
by the pairs $z_i, z_j$. It follows that $q$ vanishes on the plane
$\{ (x_1,x_2, 1)~ |~ x_1, x_2 \in \R\}$. Hence $q\equiv 0$.
\end{proof}

Let $Z \subset \R^2$ be a noncolinear triple. Then $\hat{Z}$ 
is in general position. Let $\vec{d}=(d_1, d_2, d_3)$ be an 
oriented negative natural basis for $Q_Z$. 
(See \S \ref{SectionQuadratic}). Let $T_{\vec{d}} \subset Q_Z$ denote the plane  
\[   T_{\vec{d}}~ =~ \left\{ \sum_i t_i \cdot d_i~ \left|~  \sum_i t_i~ =~ 1 \right\}. \right. \]
Let $r_{\vec{d}}: T_{\vec{d}} \rightarrow \Scal_{Z}$ denote the restriction of
the map $q \mapsto  U_q$ to $T_{\vec{d}}$.

\begin{prop} \label{Restriction}
$r_{\vec{d}}$ is a homeomorphism from $T_{\vec{d}}$ onto $\Scal_Z$.
\end{prop}

\begin{proof}
Since $\vec{d}=\{d_1, d_2, d_3\}$ is a basis for $Q_Z$, 
Lemma \ref{MissingNothing} implies that map $q \mapsto U_q$  is onto $\Scal_Z$. 
Note that $U_q=U_{q'}$ if and only if there exists $\lambda \in \R^+$
such that $\lambda \cdot q = q'$.  It follows that $r_{\vec{q}}$ is injective.
The continuity of $q \mapsto U_q$ and the inverse of its restriction
follow from linearity and the definitions of the various topologies.
\end{proof}

The plane $T_{\vec{d}} \subset \R^3$ has a canonical affine structure 
and a canonical outward normal orientation. 
In particular, the ordered set $(d_2-d_1, d_3-d_1)$ is an oriented basis
for the tangent space to $T_{\vec{d}}$ at $d_1$. 

\begin{prop} \label{RestrictionNatural}
Let $\vec{d}$ and $\vec{d}'$ be oriented negative natural bases for 
$Q_Z$. The map 
$r_{\vec{d}}^{-1} \circ r_{\vec{d}'}: T_{\vec{d}'} \rightarrow  T_{\vec{d}}$ 
is an orientation preserving affine map. 
\end{prop}

\begin{proof} 
Since $\vec{d}=(d_1, d_2,d_3)$ and   $\vec{d}'= (d_1', d_2', d_3')$
are both oriented negative degenerate bases, there exists 
there exists a $3$-cycle $\sigma \in S(\{1,2,3\})$ and 
$(\lambda_1, \lambda_2, \lambda_3)  \in (\R^3)^+$ such that for $i=1,2,3$
\[    d_i~ =~ \lambda_i \cdot d_{\sigma(i)}. \]
Define a linear map  $A: Q_Z \rightarrow Q_Z$ by setting 
$A(d_i)= \lambda_i \cdot d'_{\sigma(i)}$.
Since $\lambda_i>0$ and $\sigma$ is a 3-cycle, 
the map $A$ is orientation preserving.
The map $r_{\vec{d}} \circ r_{\vec{d}'}^{-1}$ is the restriction of the 
$A$ to the plane $T_{\vec{d}'}$. The claim follows.
\end{proof}

Let $\Ecal_Z$ denote the collection of ellipse interiors 
whose boundaries contain $Z$.

\begin{prop} \label{hInverse}
Let $Z \subset \R^2$ be a noncollinear triple, let $\vec{d}$ be a
natural basis for $Q_Z$, and let $T=T_{\vec{d}}$. 
The set $\{q \in T~ |~ U_q \in \Ecal_Z\}$ 
is a bounded, convex  subset of the plane $T$. 
In particular, $\overline{\Ecal_Z}$ is  compact, 
and $\left\{q \in T~ |~ U_q \in \overline{\Ecal_Z} \right\}$ 
is a compact convex subset of $T$.  
\end{prop}

\begin{proof}
The set $Q_Z$ is a vector subspace, and hence it follows from 
Proposition \ref{EllipseConvex} that the set of all $q$ such that
$U_q \in \Ecal_Z$ is a convex subset of $Q(\R^3)$. Thus, 
$\{q \in T~ |~ U_q \in \Ecal_Z\}$ is convex. 

Let $(\hat{z}_1,\hat{z}_2,\hat{z}_3)$ be the cyclic ordering of the elements 
of $\hat{Z}$ with respect to the orientation of $\R^3$. Let 
$\vec{d}=(d_1,d_2,d_3)$ be a corresponding oriented negative natural basis
for $Q_Z$. To show that $\left\{q \in T~ |~ U_q \in \overline{\Ecal_Z} \right\}$ 
is bounded, it suffices to show that if $U_q \in \Ecal_Z$ and $q \in T_{\vec{d}}$,
then each coordinate of $q$ with respect to $\vec{d}$ is positive.

For $\{i,j,k\}=\{1,2,3\}$, let $\sigma_{i}$ denote
the segment joining $z_j$ and $z_k$.  
Since $U_q \in \Ecal_Z$ is strictly convex and $Z \subset \partial U_q$, 
we have $\sigma_i \subset \overline{U_q}$ and $\sigma_i \cap \partial U_q= \{x_j, x_k\}$.
In particular, for each $y \in \sigma_i \setminus \hat{Z}$ we have $q(\hat{y})<0$.

Let $(t_1, t_2,t_3)$ be the coordinates of $q$ with respect to $\vec{d}$.
Since $\vec{d}$ is a negative natural base, for each $y \in \sigma_i \setminus Z$
we have $q(\hat{y})<0$. It follows that $t_i >0$ for each $i=1,2,3$. 
\end{proof}


\section{Subconics in a translation surface} 

            \label{SectionSubconicsTranslation}

In this section, we make precise the notion of subconic in a translations
surface $(X, \mu)$, we define the space of subconics in $(X, \mu)$
and its stratification, and we derive some basic properties. 

Let $\Gal(\tilde{X}/X) $ denote the 
group of covering transformations associated to the universal covering 
$p: \tilde{X} \rightarrow X$. Each covering transformation is a translation mapping 
of $(\tilde{X}, \tilde{\mu})$ and hence 
$\Gal(\tilde{X}/X)$ acts on the set of convex subsets  $\tilde{U} \subset \tilde{X}$
such that $\dev_{\mu}(\tilde{U})$ is a subconic in $\R^2$. 

\begin{defn}
A {\em subconic} in $X$ with respect to $\mu$ is an orbit of the action of $\Gal(\tilde{X}/X)$ 
on the collection of convex subsets $\tilde{U}$ of $\tilde{X}$ such
that $\dev_{\mu}(\tilde{U})$ is a subconic in $\R^2$. 
\end{defn}

There is a natural one-to-one correspondence between the 
subconics $U$ in $X$ and equivalence\footnote{Two immersions 
$f: K \rightarrow X$, $f': K' \rightarrow X$ are  equivalent iff their exists 
isometry $g: K \rightarrow K'$ so that $f' \circ g=f$.} classes  of isometric immersions 
from a subconic $U' \subset \R^2$ into $X$. Indeed, each  
isometric immersion $f: U' \rightarrow X$ lifts to an immersion 
$\tilde{f}: U' \rightarrow \tilde{X}$. The lift is unique up to the action of
$\Gal(\tilde{X}/X)$.
On the other hand, given a subconic $\tilde{U} \subset \tilde{X}$,
the map $p \circ \dev_{\mu}^{-1}$ restricted to $\dev_{\mu}(\tilde{U})$ gives
an immersion.

Each representative $\tilde{U}$ of a subconic $U$ will be called a {\em lift} of $U$.
In most cases, a subconic in $X$ is determined by the image a lift under $p$.  
For example, if $X$ is simply connected, then each subconic has a unique 
lift. In this case, we will identify the singleton set with the element
that it contains.

On the other hand, there are situations in which one must keep track of 
the orbit in the universal cover (or, equivalently, the immersion). 
For example, if $X$ is the once-punctured torus, $(\R^2 \setminus \Z^2)/ \Z^2$, 
then the set $X$ is itself the image of two ellipses 
that do not differ
by an element of $\Gal(\tilde{X}/X)$.\footnote{We thank the referree for pointing
out a similar example.} 

In practice, we will be working with a particular lift $\tilde{U}$ in the universal
cover, and we will regard the image $p(\tilde{U})$ as a subconic in $X$. 

We will use the terminology used in the classification of subconics in the plane 
to describe subconics in $X$ with respect to $\mu$.

\begin{remk}  \label{Cylinder}
An isometric embedding of a cylinder, $[a,b] \times \left( \R / c \cdot \Z \right)$,
corresponds to a strip in  a translation surface $(X, \mu)$.
\end{remk}

If $X$ is simply connected, the developing map furnishes a natural topology
on $\Scal(\tilde{X}, \mu)$.  
A set $\tilde{{\mathcal U}}\subset   \Scal(\tilde{X}, \tilde{\mu})$
is said to be open if and only if $\dev_{\mu}({\mathcal U})$ is open
in $\Scal(\R^2)$.  In other words,  $\tilde{{\mathcal U}}$ is open if and only
if the associated classes of quadratic forms constitute an open set 
in $Q(\R^3)/ \R^+$. 

More generally, we endow $\Scal(X, \mu)= \Scal(\tilde{X}, \tilde{\mu})/ \Gal(\tilde{X}/X)$ 
with the quotient topology.  
The action of the deck group $\Gal(\tilde{X}/X)$ on $\Scal(\tilde{X}, \tilde{\mu})$ 
is discontinuous, and hence the quotient map 
$p_*:  \Scal(\tilde{X}, \tilde{\mu}) \rightarrow \Scal(X, \mu)$
is a covering map with deck group $\Gal(\tilde{X}/X)$.

Roughly speaking, the space $\Scal(X, \mu)$ has a natural stratification determined
by the number of points in the boundary, $\partial \tilde{X}$,
that are met by a lift $\tilde{U}$. To be more precise we 
make the following definition.

\begin{defn}
 We say that a subset  $A \subset \R^d$  has {\em maximal span} iff there exists a 
subset $B \subset A$ such that the 
span of $\{b-b'~ |~ ~ b, b' \in B\}$ has dimension at least
as large as either $d$ or ${\rm Card}(A)-1$.
\end{defn}

If ${\rm Card}(A) \leq d$, then $A$ has maximal span iff $A$ is in general 
position.  If ${\rm Card}(A) > d$, then $A$ has maximal span iff the 
set of differences $\{a-a'~ |~ a,a' \in \R^d\}$ spans $\R^d$. 

\begin{defn}
For $n\in \Z^+$, define the {\em $n$-stratum}, $\Scal_n(x, \mu)$,  to be
the collection of all subconics $U \subset X$ 
such that ${\rm Card}(\partial \tilde{U} \cap \partial \tilde{X}) \geq n$
and  $\dev_{\mu}(\partial \tilde{U} \cap \partial \tilde{X})$ has maximal span in $\R^2$.
\end{defn}

If one component of the boundary of a strip $U$ contains at least three points in 
$\partial X$ while the other other component contains none, then 
the strip does not belong to $\Scal_n(X, \mu)$ for any $n$. Indeed, 
in this case, the span of $\dev_{\mu}(\partial \tilde{U} \cap \partial \tilde{X})$ 
is 1-dimensional.

\begin{prop} \label{TopStratumProjective}
The space $\Scal(X, \mu) \setminus \Scal_1(X,\mu)$ 
is a five dimensional real projective manifold. 
\end{prop}

\begin{proof}
Since $\dev_{\mu}$ is a local embedding, it follows that the map 
$U \mapsto \dev_{\mu}$ is a local embedding from  $\Scal(\tilde{X}, \mu)$ to
$\Scal(\R^2)$.  The set $\Scal(\tilde{X}, \mu) \setminus \Scal_1(\tilde{X},\mu)$ is mapped
to an open subset of $\Scal(\R^2)$.  
Define $f: \Scal(\R^2) \rightarrow  PQ(\R^3)$ by
$f(U) =  \{q~ |~ U_{\pm q}= \dev_{\mu}(U)\}$. Note that $f$ is
a local embedding.  

The action of the translations on $X$ is $\dev_{\mu}$-equivariant 
action of a subgroup of translations on $\R^2$. In turn this action 
is $f$-equivariant to an action of a subgroup of $PGL(Q(\R^3))$. 
\end{proof}


\section{Subconics in coverings of precompact surfaces with finite frontier}
 \label{SectionCovering}

In this section and the ones that follow, we make the following assumption. 
\begin{ass} \label{FiniteFrontierAssumption}
We consider only translation surfaces $(X, \mu)$ 
that cover a precompact translation surface with 
nonempty frontier.
\end{ass}
 
This assumption limits the types of subconics that can be immersed in $X$. 
Recall that an embedded cylinder is a strip. See Remark \ref{Cylinder}.

\begin{prop} \label{EllipseStrip}
Each subconic $U \subset X$ is either a strip or an ellipse interior.
\end{prop}

\begin{proof}
Let $\tilde{U}$ be a lift of $U$ to the universal cover $\tilde{X}$. 
It follows from Proposition \ref{Inradius} and Assumption \ref{FiniteFrontierAssumption}
that the inradius of $\tilde{X}$ is finite. 
The only nontrivial subconics that have finite inradius are strips 
and ellipse interiors.
\end{proof}

In order to deal with a strip $U \subset X$, it will prove convenient to 
construct a canonical `maximal' neighborhood $M$ of $U$ such that the 
restriction $\dev_{\mu}|_M$ to $M$ is an injective map.

\begin{defn}  \label{CanonicalNeighborhood}
Let $M$ be the collection of all $x \in X$ 
such that there exists $y \in U$ and a geodesic segment $\alpha$
joining $x$ to $y$ such that the vector $\dev_{\mu}(x)- \dev_{\mu}(y)$
is orthogonal to the radical of $q$ where $U_q = \dev_{\mu}(U)$.
We will call $M$ the {\em canonical neighborhood} of $U$. 
\end{defn}

\begin{figure}[h]   
\begin{center}

\includegraphics[totalheight=2in]{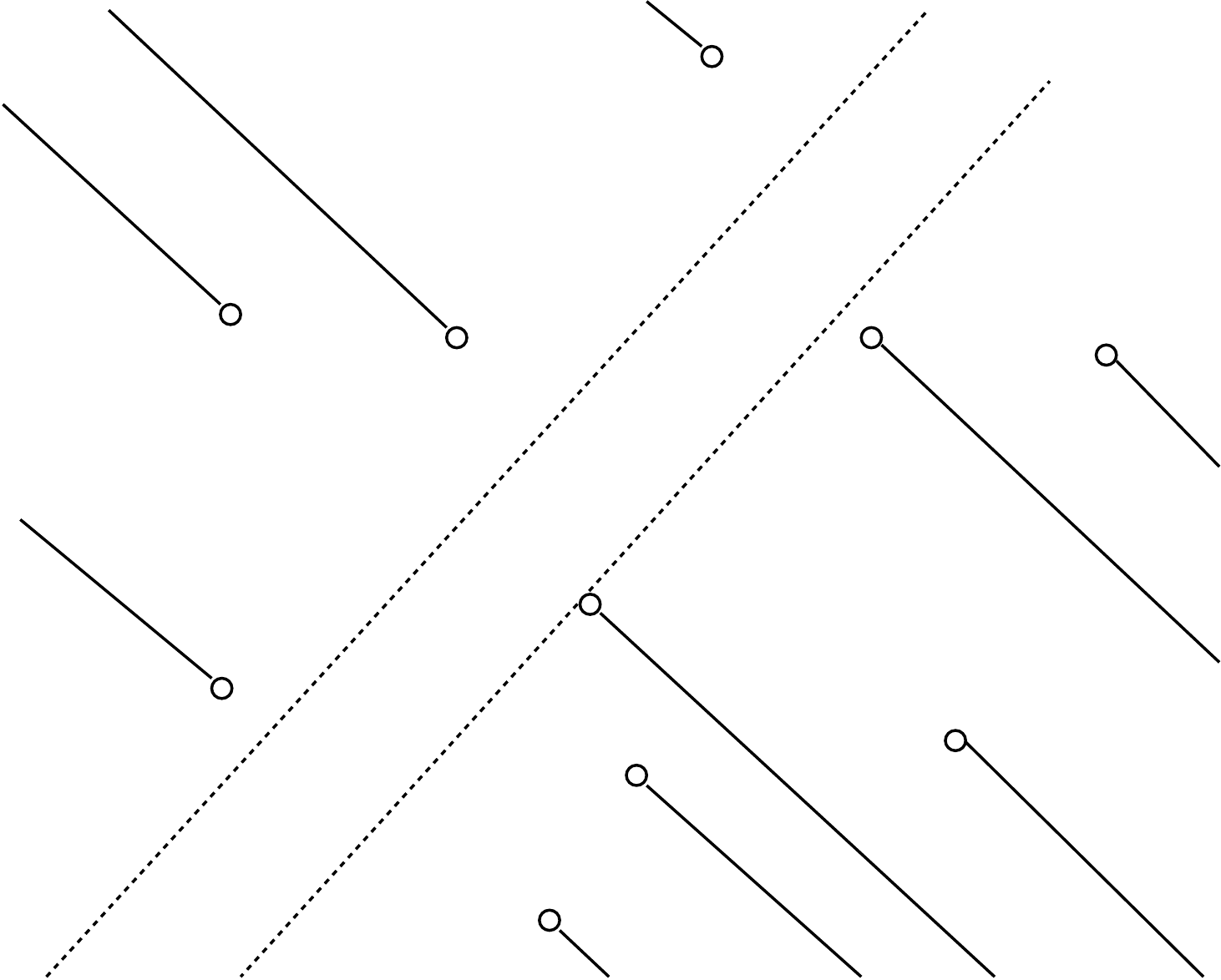}

\end{center}
\caption{\label{CanonicalFigure} The $\dev_{\mu}$ image of a 
canonical neighborhood of a strip. The circles 
indicate points that belong to $\dev_{\mu}(\partial X \cap \overline{M})$.}
\end{figure}

The following is elementary.

\begin{prop} \label{MaximalStrip}
If $U \subset X$ is a strip, then there exists a unique strip 
that contains every strip that contains $U$.  If $\tau:X \rightarrow X$ 
is a nontrivial translation and $\tau(U)=U$, then for each connected 
component $C$ of $\partial U$, the set $C \cap \partial X$ is infinite and  
$\tau(C \cap \partial X)= C \cap \partial X$.
\end{prop}

We will call the strip described in Proposition \ref{MaximalStrip}
the {\em maximal strip} containing $U$.

\begin{prop}  \label{TranslationExists}
If $(X, \mu)$ is simply connected and $U \subset X$ is a maximal strip, 
then there exists a nontrivial translation $\tau: X \rightarrow X$ such that
$\tau(U)=U$. 
\end{prop}

\begin{proof}
Let $p: X \rightarrow Y$ be as in Assumption \ref{FiniteFrontierAssumption}.
Since $X$ is simply connected, $p$ is the universal covering.
Since $\partial Y$ is finite, $Y$ is precompact, and hence $p(U)$ is precompact. 
Since $U$ is a strip, it is not precompact  (Proposition \ref{EllipseBounded}). 
Therefore, there exists $x, x' \in U$ and a nontrivial deck transformation 
$\tau:  X \rightarrow X$ such that $\tau(x)=x'$. Because $U$ is maximal 
and $\tau(\partial X) = \partial X$,  we find that $\tau(U)$ is also maximal. 

The intersection $\tau(U) \cap U$ is a strip. 
Since $U \cap \tau(U)$ belongs to both $U$ and $\tau(U)$, if follows from
the uniqueness of maximal strips---Proposition \ref{MaximalStrip}---that 
$U =\tau(U)$.      
\end{proof}

By combining Proposition \ref{MaximalStrip} and \ref{TranslationExists} 
 we obtain the following.

\begin{coro} \label{StripInfinite}
Each maximal strip (resp. cylinder) belongs to $\bigcap_n S_n(X, \mu)$.
\end{coro}

\begin{prop} \label{RigidSubconicsDiscrete}
 $\Scal_5(X, \mu)$ is a discrete and countable subset of $\Scal(X, \mu)$.
\end{prop}

\begin{proof}
By assumption there exists a precompact translation surface $(X, \nu)$ with 
finite frontier and a covering $p:X \rightarrow Y$ such that 
$\mu = p \circ \nu$. It suffices to assume that $p$ is a universal covering.
Note that $p$ extends uniquely to a continuous map from $\overline{X}$
to $\overline{Y}$ that we will abusively call $p$. 

Let $U$ be a subconic such that $\partial U \cap \partial X$ 
has at least five points.  Given $\delta$, define the $\delta$-neighborhood 
of $U$ by 
\[  N_{\delta}~ =~ 
   \left\{ x \in \overline{X}~ |~ {\rm dist}\left(y,  U \right) < \delta \right\}.
\] 
Since  $\partial Y$ is finite, 
there exists $\epsilon>0$ such that 
$p(N_{\epsilon})~ \cap \partial Y=~  p(\overline{U}) \cap \partial Y$.
Thus, we find that
\begin{equation} \label{Epsilon}
 N_{\epsilon} \cap \partial Y~ =~  \partial U \cap \partial X. 
\end{equation}

Let $q  \in Q(\R^3)$ be such that $\dev_{\mu}(U)=U_q$.
Let $M \subset \R^2$ denote the image of $N_{\epsilon}$ under the developing map.
Since $\dev_{\mu}$ is an open mapping, the set $M$ is open. 
By Proposition \ref{EllipseStrip}, $U$ is either an ellipse interior or a strip.
In each of these cases, we will show that $U$ is isolated in $\Scal_5(X, \mu)$.

Suppose that $U$, and hence $U_q$, is an ellipse interior. Let
$x_0 \in U_q$. Let $A$ be the set of all $r \in Q(\R^3)$
such that either $r(\widehat{x_0})\geq 0$ or 
there exists $x \in \partial M$ such that $r(\widehat{x}) \leq 0$. 
Since $\partial M$ is compact, the set $A$ is closed. 
Since $q(x_0)<0$ and $U_q \subset M$, the form $q$ does not belong to $A$.

Let $U'$ be a subconic in $X$ such that $\partial U \cap \partial X$ 
has at least five elements. Let $q' \in Q(\R^3)$ so that $U_{q'}= \dev_{\mu}(U')$.
We claim that if $U' \neq U$, then $q' \in A$. Note that if $q'(x_0)\geq 0$
held, then we would have $q' \in A$. Thus, we assume, without loss of generality,
that $x_0 \in U'$. 

It follows from (\ref{Epsilon}) that 
each element of $\dev_{\mu}(\partial U' \cap \partial X)$ 
either belongs to $\dev_{\mu}(\partial U \cap \partial X)$ or to the 
complement of $M$.  If $\dev_{\mu}(\partial U' \cap \partial X)$ were a subset 
of $\dev_{\mu}(\partial U \cap \partial X)$, then since  $\dev_{\mu}(\partial U' \cap \partial X)$ contains at least five points, 
Proposition \ref{Dimension} would imply that $U'= U$. 
In other words, if $U' \neq U$, then there exists 
$x \in \partial U' \cap \partial X$
that does not belong to $M$. In particular,  $x \in \overline{U}_{q'}$.

Since $x$ and $x_0$ both belong to the convex set $\overline{U}_{q'}$,
there exists a path $\alpha: [0,1] \rightarrow \R^2$  with $\alpha(0)=x_0$
and $\alpha(1)=x$.  Let $t_0= \sup \{ t~ |~ \alpha(t) \in M\}$. 
Because $x_0 \in M$ and $x \notin M$, we have that $\alpha(t_0) \in \partial M$.
But $q(\alpha(t_0)) \leq 0$ and hence $q \in A'$ as desired.  

The image of $A$ under the map $q \rightarrow U_q$ is a closed subset of 
$\Scal(\R^2)$ that contains $\dev_{\mu}(\Scal_{5}(X,\mu) \setminus \{U\})$
but does not contain $\dev_{\mu}(U)$.  Hence, if $U \in \Scal_5(X, \mu)$ 
is an ellipse interior, then it is an isolated point of $\Scal_5(X, \mu)$. 

If $U \in \Scal_5(X, \mu)$ is a strip, 
then it follows from Propositions  \ref{MaximalStrip}
and \ref{TranslationExists}
that there exists a nontrivial planar translation $\tau$ such that 
$\tau(Z)=Z$ where $Z=\dev_{\mu}(\partial U \cap \partial X)$.   
Since $\tau(Z \cap \partial U)= Z \cap \partial U$,
the length of each component  of  $\partial U \setminus Z$ is at most the translation 
distance $|\tau|$ of $\tau$. 

Let $w$ denote the distance between the boundary components of the strip $U_q$.
Let $\ell$ be the `center-line' of $U_q$, namely  
$\ell:=\{x \in U_q~ |~ {\rm dist}(x, \partial U_q)= w/2\}$.  
Let $\sigma$ be a compact connected subset of $\ell$ with
length 
\[  |\sigma|~ >~  \frac{\epsilon+ w}{\epsilon} \cdot |\tau|. \]

Let $A \subset Q(\R^3)$ be the set consisting of all $r$ such that 
for some  $x \in \sigma$, we $q(\hat{x}) \geq 0$.  
Note that $A$ is a closed subset of $Q(\R^3)$. Indeed, given a sequence 
$\{r_n\} \subset A$, let $x_n \in \sigma$ such that $r_n(x_n) \geq 0$.
Since $\sigma$ is compact, a subsequence of $x_n$ converges 
to some $x \in \sigma$. Thus, if $r_n$ converges to $r$, then $r(x) \geq 0$.   
 
For each $x \in \sigma$, we have $q(x)<0$, and hence  $q \notin A$.
Let $U' \neq U$ be a subconic such that $\partial U \cap \partial X$
contains at least five elements. Let $q' \in Q(\R^3)$ be such that 
$U_{q'}= \dev_{\mu}(U')$. It suffices to show that $q' \in A$.  
Indeed, $A$ is closed and $q \notin A$.

An argument similar to the one given in the case of ellipse interiors
shows that since $q \neq q'$ there exists $x'$ such that ${\rm dist}(x', U)=\epsilon$ 
and such that $q(x)=0$.  Since $U_{q'}$ is convex,
the set $\sigma'= \ell \cap \overline{U_{q'}}$ is a 
compact connected set.  If $\sigma'=\emptyset$, then $q'(x)=0$ for all $x \in \ell$,
and so $q' \in A$. Therefore, it suffices to assume that $\sigma'=\emptyset$. 

Let $T$ be the convex hull of $x'$ and the endpoints of $\sigma'$.  
Observe that $T$ is a (perhaps degenerate) triangle.
Since $\overline{U_{q'}}$ is convex, the triangle $T$ is a subset of $\overline{U_{q'}}$. 
Since $\sigma'$ is parallel to $\partial U$, we may use a similar triangles 
argument to show that the segment $\alpha= T \cap \partial U$ has
length 
\[  |\alpha|~ =~ |\sigma'| \cdot \frac{\epsilon}{\epsilon+ w}. \]
Since $U \cap Z =\emptyset$, the interior of $T$ does not intersect $Z$. 
It follows that $|\alpha| < |\tau|$, and hence $|\sigma'|< |\sigma|$.
Thus, there exists $x \in \sigma$ such that $q'(x)\geq 0$.
Therefore  $q' \in A$ as desired.  
\end{proof}

Proposition \ref{RigidSubconicsDiscrete} leads us to  make the following definition. 

\begin{defn}
The  elements of $\Scal_5(X, \mu)$ will be called  {\em rigid} subconics.
\end{defn}

The following proposition ensures that each polygonal $2$-cell in the cellulation
of $\Scal_3(X, \mu)$ has finitely many vertices. (See \S \ref{SectionCellComplex}).

\begin{prop}  \label{FiniteEllipses} 
Let $V \subset X$ be open and nonempty. 
The set of all rigid subconics that contain $V$ is finite. 
\end{prop}

\begin{proof}
By Proposition \ref{RigidSubconicsDiscrete}, it suffices to show that 
the set of all rigid subconics that contain $V$ is compact.
 Without loss of generality, we may assume that $(X,\mu)$
is a universal covering of a precompact translation surface that has finite frontier.

Let $U_n$ be a sequence of rigid subconics containing $V$.
Since $V$ is nonempty and each $U_n$ is convex, the union 
$W = \bigcup_n U_n$ is star convex. Thus, the restriction of $\dev_{\mu}$ 
to $W$ is injective. 

Let $\|\cdot \|$ be a norm on the finite dimensional vector space $Q(\R^3)$.
For each $n$, let $q_n$ be such that $U_{q_n}=\dev_{\mu}(U_n)$ and $\|q\|=1$.
Since the unit sphere is compact, the sequence $q_n$ has a limit point $q$.
Without loss of generality, the sequence  $q_n$
converges to $q$ for otherwise take a subsequence.    
We have $U_q \subset \dev_{\mu}(W)$, and hence $\dev_{\mu}|_W^{-1}(U_q)$
is a subconic in $X$. It suffices to show that $U$ is rigid.

We first claim that $U_q$ is either an ellipse interior or a strip. 
To prove this, we will eliminate the other possibilities:  
$U_q$ is empty, is the complement of a line, is a parabola, or is the plane.
For each $n$, we have $\dev_{\mu}(V) \subset U_{q_n}$. Thus, since $V$ is open
and $q_n \rightarrow q$, we have $V \subset U_q$ and hence $U_q$ is nonempty.   
By Proposition \ref{Inradius}, $X$ has finite inradius $R$. Thus for each $n$, the 
inradius of  $U_{q_n}$  is at most $R$, and hence the inradius of $U_q$
is at most $R$. This eliminates the remaining cases.

Define
\[  M_{\epsilon}~ =~ \left\{ x~ |~ q(\hat{x}) < \epsilon \right\}. \]
Since $q_n \rightarrow q$, for each $\epsilon>0$ there exists $N_{\epsilon}$ such that if 
$n >N_{\epsilon}$, then $\overline{U}_{q_n} \subset M_{\epsilon}$. 
In particular, if $n> N_{\epsilon}$, then the set
$Z_n = \dev_{\mu}(\partial U_n \cap \partial X)$ is a subset of $M_{\epsilon}$.

Suppose that $U_q$ is an ellipse interior.  Since $q_n \rightarrow q$,
Proposition \ref{EllipseOpenFrontier} implies that---by passing to a tail
if necessary---we may assume that for each $n$, 
the subconic $U_{q_n}$ is an ellipse that lies in $M_1$. In particular, 
$\dev_{\mu}(W) \subset M_1$. Since $\overline{M_{1}}$ is compact, $\partial X$ is discrete,
and $\dev_{\mu}|_W$ is a homeomorphism onto its image, 
the set $F=\dev_{\mu}(\overline{W} \cap \partial X)$
is finite. 

For each $n$, let $x_n^1 \in Z_n$. Since $Z_n \subset Z$,
there exists an infinite set  $A^1 \subset \N$ and $c^1 \in F$
such that if $n \in A^1$, then $x_n^1=c^1$. For each $\epsilon>0$, 
we have $c^1 \in M_{\epsilon}$ and hence $c^1 \in \dev_{\mu}(\partial U \cap \partial X)$.
Since ${\rm Card}(Z_n)>1$, for each $n \in A^1$, there exists 
$x_n^2 \in Z_n$ such that $x_n^2 \neq c^1$.  Since $F$ is finite, there exists an infinite set 
$A^2 \subset A^1$ and  $c^2 \in F$ such that if $n \in A^2$, then $x_n^2=c^2$.
We have $c^2 \in \dev_{\mu}(\partial U \cap \partial X)$. 
By continuing in this way, we find $\{c^1,c^2, c^3, c^4, c^5\} \subset 
 \dev_{\mu}(\partial U \cap \partial X)$.  Hence $U \in \Scal_5(X, \mu)$.

If $U$ is a strip, then by Proposition \ref{MaximalStrip},
$U$ belongs to a maximal strip $U'$.
For each $\epsilon>0$, we have $Z_n \subset M_{\epsilon}$. Thus for each $\epsilon>0$, the set $\dev_{\mu}(U') \subset M_{\epsilon}$.
Hence $\dev_{\mu}(U') \subset \dev_{\mu}(U)$ and thus $U'=U$. 
Thus, by Corollary \ref{StripInfinite}, $U$ is a rigid conic.     
\end{proof}


\section{The homotopy type of the space of ellipse interiors} \label{HomotopySection}

Let $\Ecal(X, \mu)$ denote the space of ellipse interiors in a translation surface $(X, \mu)$.

\begin{prop}
The space $\Ecal(X, \mu)$ is homotopy equivalent to $X$. 
\end{prop}

\begin{proof}
Given an ellipse interior $U$ in $X$, let $\tilde{c}(U)$ be the center of a 
lift of $U$ to $\tilde{X}$. Define a map $c: \Ecal(X, \mu) \rightarrow X$
by setting $c(U)= p(\tilde{c}(u))$.  Given a point $x \in X$, let $\tilde{D}(x)$ be 
largest disk contained in $\tilde{X}$ with center at some $\tilde{x} \in p^{-1}(x)$. 
(Since $X$ is an open manifold, a disk exists, and by Assumption \ref{FiniteFrontierAssumption}
and Proposition \ref{Inradius} there is a largest such disk).  
Let $D(x)$ be the orbit of $\tilde{D}(x)$ under $\Gal(\tilde{X}/X)$.
It is straightforward to show that the maps $D: X \rightarrow \Ecal(X, \mu)$
and $c$ determine a homotopy equivalence. 
\end{proof}

\begin{prop}
There is a deformation retraction from $\Ecal(X, \mu)$ onto $\Ecal_3(\tilde{X}, \tilde{\mu})$. 
\end{prop}

\begin{proof}
We first define the retraction under the asumption that $X$ is simply connected.

Let $X$ be simply connected.
Let $\Ecal_n(X, \mu)$ denote the collection of ellipses $U \subset X$
such that $\partial X \cap \partial U$ contains at least $n$ points.
The desired retraction is a concatenation of three retractions:
from $\Ecal(X, \mu)$ to $\Ecal_1(X, \mu)$, from $\Ecal_1(X, \mu)$ to $\Ecal_2(X, \mu)$,
and from $\Ecal_2(X, \mu)$ to $\Ecal_3(X, \mu)$. 

The first retraction consists of dilating the ellipse interior 
until the boundary meets $\partial X$.
To be precise, let $\tau$ be the planar translation that sends the center of 
$\dev(U)$ to the origin. Let $q$ be the quadratic form such that 
$U_q= \tau(\dev(U))$.  Since the center of $U_q$ is the origin,
we have 
\[    q(v)~ =~ a \cdot v_1^2 + \underline{q}(v_2,v_3). \]
where $a <0$ and $\underline{q}$ is a positive definite form on $\R^2$.
Define 
\[    q_t(v)~ =~  a \cdot v_1^2~ +~ (1-t) \cdot  \underline{q}(v_2,v_3). \]
Note that the area of $U_{q_t}$ increases to infinity at $t$ tends to 1.
By Assumption \ref{FiniteFrontierAssumption}
and Proposition \ref{Inradius}, the supremum, $t_0$, of the set of
$t$ such that $\tau_U^{-1}(U_{q_t})$ is the $\dev$-image of a lift of a subconic
is less than one. By setting 
\begin{equation} \label{f1}
   f_1(s,U)~ =~  \dev_{\mu}^{-1} \left( \tau^{-1} \left( U_{q_{(s/t_0)}} \right) \right) 
\end{equation}
we obtain a retraction  $f_1: [0,1] \times \Ecal(X, \mu) \rightarrow  \Ecal_1(X, \mu)$.

The deformation retract from $\Ecal_1(X, \mu)$ onto the 
$\Ecal_2(X, \mu)$ is defined by dilating and translating the center while
maintaining contact with $\partial X$.  To be precise, given an ellipse 
$U$ and $x \in \partial U \cap \partial X$, let $\tau$ and 
$q$ be as before. Let $\eta$ be a linear form on $\R^3$
defined as the differential of $q$ at $\widehat{\dev_{\mu}(\tau(x))}$ 
In particular, $\eta\left(\widehat{\dev_{\mu}(\tau(x))}\right)=0$.
Define 
\begin{equation} \label{qt}    q_t~ =~  q~ -~  t \cdot \eta^2. 
\end{equation}
It follows from Assumption \ref{FiniteFrontierAssumption}
and Proposition \ref{Inradius} that the supremum, $t_0$, of the set of
$t$ such that $\tau_U^{-1}(U_{q_t})$ is the $\dev$-image of a lift of a subconic
is finite.   By defining $f_3$ as in (\ref{f1}),
we obtain a retraction $f_2: [0,1] \times \Ecal_1(X, \mu) \rightarrow  \Ecal_2(X, \mu)$.

The deformation retract from $\Ecal_2(X, \mu)$ onto the 
$\Ecal_3(X, \mu)$ is defined by dilating and translating the center while
maintaining contact with two points in $\partial X$.  To be precise, given an ellipse 
$U$ and $x_1, x_2 \in \partial U \cap \partial X$, let $\tau$ and 
$q$ be as before. Let $\eta$ be a linear form on $\R^3$ whose kernel
is spanned by $\widehat{\dev_{\mu}(\tau(x_1))}$ and $\widehat{\dev_{\mu}(\tau(x_2))}$ 
and so that the value at the midpoint of these two vectors equals 1. 
Define $q_t$ as in (\ref{qt}). 
It follows from Assumption \ref{FiniteFrontierAssumption}
and Proposition \ref{Inradius} that the supremum, $t_0$, of the set of
$t$ such that $\tau_U^{-1}(U_{q_t})$ is the $\dev$-image of a lift of a subconic
is finite. By defining $f_3$ as in (\ref{f1}),
we obtain a retraction $f_3: [0,1] \times \Ecal_1(X, \mu) \rightarrow  \Ecal_2(X, \mu)$.

By concatenating $f_1$, $f_2$ and $f_3$, we obtain a retract 
$f: \Ecal(X, \mu) \rightarrow \Ecal_3(X, \mu)$.

Suppose that $\tilde{X}$ is the universal cover of a translation surface $(X, \mu)$.
Let $\sigma: \tilde{X} \rightarrow \tilde{X}$ be a translation mapping. 
Since the translation $\tau$ associated to $U$ differs from the translation
associated to $\sigma(U)$ by $\sigma$, we find that 
$f_i(s, \sigma(U)) =  f_i(s, U)$ for each $U \in \Ecal_{i-1}$.
Since each covering transformation is a translation, the 
retraction $f$ descends to a homotopy retraction from 
$\Ecal(X, \mu)$ onto $\Ecal_3(X, \mu)$.   
\end{proof}

\begin{coro}  \label{Connected}
If  $X$ is connected, then $\Scal_3(X, \mu)$ is connected.
\end{coro}

\begin{coro} \label{SimplyConnected}
$\Ecal(X,\mu)$ is simply connected if and only if $X$ is simply connected.
\end{coro}

The deck group ${\rm Gal}(\tilde{X}/X)$ acts on $\tilde{X}$ as translations. 
In particular, ${\rm Gal}(\tilde{X}/X)$ acts on each stratum $\Scal_n(\tilde{X}, \tilde{\mu})$.
The action is free and discontinuous, and hence the natural projection map 
$p_*: \Scal_n(\tilde{X}, \tilde{\mu}) \rightarrow \Scal_n(X,\mu)$
is a covering. 

\begin{coro} \label{Covering}
The map $p_*: \Scal_n(\tilde{X}, \tilde{\mu}) \rightarrow \Scal_n(X,\mu)$
is a universal covering.
\end{coro}

From henceforth we assume that $X$ is connected.


\section{A two dimensional cell complex} 
\label{SectionCellComplex}

In this section we will assume that 
$(X,\mu)$ is the universal cover of a precompact translation surface
with finite frontier.  Here we show that the set, $\Scal_3(X, \mu)$ 
is naturally a $2$-dimensional polyhedral  complex.
Moreover, $\Scal_4(X, \mu)$ is the 1-skeleton and 
${\mathcal S}_{5}(X, \mu)$ is the 0-skeleton. 
To simplify the exposition, we will sometimes abbreviate $\Scal_n(X,\mu)$
by $\Scal_n$. 

Given $Z \subset \partial X$, let $\Scal_Z(X, \mu)$ denote the 
set of all subconics $U$ such that $Z= \partial U \cap \partial X$
and such that $Z$ intersects each component of $\partial U$.
Note that $\overline{\Scal}_Z(X, \mu)$ is the set of 
all subconics such that $Z \subset \partial U \cap \partial X$.
We will abbreviate  $\Scal_Z(X, \mu)$ by $\Scal_Z$
if the context makes the choice of $(X, \mu)$ clear.

Abusing notation slightly, we will let $\dev_{\mu}$ denote the
map from $\Scal(X, \mu)$ to $\Scal(\R^2)$ defined by 
$U \mapsto \dev_{\mu}(U)$.

\begin{lem} \label{RestrictionInjective}
If $\Card(Z) \geq 2$, then the restriction of $\dev_{\mu}$ to 
$\overline{\Scal}_Z$ is a homeomorphism onto its image in $\Scal(\R^2)$. 
\end{lem} 

\begin{proof}
Suppose that $\Scal_Z \neq \emptyset$ and let $\{z, z' \} \subset Z$.
Each $U \in \Scal_Z$ contains the segment $\sigma$ joining $z$ and $z'$.
It follows that the union, $W$, of all $U \in \Scal_Z$ is star convex
with respect to, for example, the midpoint of $\sigma$. In particular,
the restriction of the developing map to $W$ is injective. Thus, $\dev_{\mu}$
determines an injection from $\Scal_Z$ into  $\overline{\Ecal(\R^2)}$.
Hence since $\dev_{\mu}$ defines the topology of $\Scal(X, \mu)$,
the restriction is a homeomorphism onto its image. 
\end{proof}

Let $Z \subset \partial X$. If there exists a convex subset 
$P \subset \overline{X}$ with nonempty interior 
such that $\dev_{\mu}(P)$ equals the convex
hull of $\dev_{\mu}(Z)$, then we will say that 
{\em $Z$ defines the polygon $P$}.\footnote{One may na\"ively 
regard the set $P$ as the convex hull of $Z$. See Remark \ref{ConvexityRemark}.}

We will say that a set of points $A \subset \overline{X}$ is {\em noncollinear}
if and only if $\dev_{\mu}(A)$ is not a subset of some line. 

\begin{prop} \label{ZPolygon}
If $Z \subset \partial X$ contains three noncollinear points
and $\Scal_Z(X, \mu) \neq \emptyset$, then $Z$ defines a polygon. 
\end{prop}

\begin{proof}
Let $U$ be a subconic such that $Z= \partial U \cap \partial X$.
By Proposition \ref{EllipseStrip}, the set $\overline{U}$ is convex.
Thus, $\dev_{\mu}|_{\overline{U}}$  is injective. Since $Z$ 
contains three noncollinear points, the convex hull $H$ 
of $\dev_{\mu}(Z)$ has nonempty interior. The set 
$P= \dev_{\mu}|_{\overline{U}}^{-1}(H)$ is the desired set. 
\end{proof}

\begin{remk}
Note that $\Scal_Z \neq \emptyset$ does not imply that $Z$ defines a polygon. 
For example, let $Z$ be the intersection of $\partial X$ and a single boundary
component of a strip. On the other hand, the assumption that $Z$ defines 
a polygon does not imply that $\Scal_Z \neq \emptyset$.  For example, 
suppose that $U$ is a subconic such that $\dev_{\mu}(Z)$ is the unit disc
and $\dev_{\mu}(\partial U \cap \partial X)$ is the set of roots of $z^6-1$.
(Here we have made the usual identification of $\R^2$ with the complex plane). 
If $Z$ is the set of roots of $z^3-1$, then  $\Scal_Z= \emptyset$. 
\end{remk}

\begin{prop}  \label{ScalCompactness}
If $Z$ defines a polygon, then the closure $\overline{\Scal_Z(X, \mu)}$ is compact. 
\end{prop}

\begin{proof}
By Proposition \ref{EllipseStrip},  we have 
\[ \dev_{\mu}\left(\overline{\Scal_Z(X, \mu)} \right)~ \subset~
   \overline{\Ecal_{\widehat{\dev_{\mu}(Z)}}(\R^2)}.  \]
By Proposition \ref{hInverse}, the set $\overline{\Ecal_{\widehat{\dev_{\mu}(Z)}}(\R^2)}$ 
is compact. 
Thus, by Proposition  \ref{RestrictionInjective} it suffices to show that 
$\dev_{\mu}\left(\overline{\Scal_Z(X, \mu)} \right)$ is closed in 
$\Scal_{\dev_{\mu}(Z)}(\R^2)$.

Let $\{U_n\} \subset \overline{\Scal}_Z(X, \mu)$ be a sequence 
such that $\{\dev_{\mu}(U_n)\}$ converges to $V \in \Scal_{\dev_{\mu}(Z)}(\R^2)$. 
Each $U_n$ contains the interior of the polygon $P$ defined by $Z$. 
In particular, the intersection  $\bigcap_n U_n$
is nonempty, and therefore the union $W = \bigcup_n U_n$ is star convex. 
Therefore, the restriction $\dev_{\mu}|_W$ is an injection
and the image is the union $\dev_{\mu}(W)=\bigcup_n \dev_{\mu}(U_n)$.

Let $q \in Q(\R^3)$ such that $U_q= V$. Since $\dev_{\mu}(U_n)$
converges to $V$, there exists sequence $q_n$ converging to $q$ such that
$U_{q_n}= \dev_{\mu}(U_n)$.  If $x \in \R^2$ and $q(\hat{x})<0$, 
then there exists $N$ such that 
for all $n>N$, we have $q_n(\hat{x})<0$.  Thus, $V \subset  \dev_{\mu}(W)$ and 
$U=\dev_{\mu}|^{-1}(V)$ is a subconic belonging to $\Scal_Z(X,\mu)$.  
\end{proof}

If $Z$ defines a triangle, then $W=\dev_{\mu}(Z)$ is a noncollinear triple.
Thus we can apply the results of \S \ref{SectionPlanarSubconics}.
In particular, let $(d_1,d_2,d_3)$ be an oriented negative natural basis 
for $Q_{\widehat{\dev_{\mu}(Z)}}$ and let $T$ be the plane  consisting
of linear combinations $\sum_i t_i \cdot d_i$ with $\sum t_i=1$. 
By Lemma \ref{Restriction} the restriction of $q \mapsto U_q$ to $T$ 
is a homeomorphism onto $\Scal_{\dev_{\mu}(Z)}$. Let $h: \Scal_{\dev_{\mu}(Z)} \rightarrow T$ 
denote the inverse of this homeomorphism.

Define $f: \overline{\Scal_Z(X,\mu)} \rightarrow T$ by 
\[    f~ =~  h \circ  \dev_{\mu}. \] 
Since $\dev_{\mu}$ and $h$ are homeomorphisms onto their respective 
images, the map $f$ is a homeomorphism onto its image. 

In general, if $Z$ defines a polygon, then 
there exists a triple $Z' \subset Z$ that defines a triangle.  
Note that $\Scal_Z \subset \Scal_{Z'}$
and hence one may restrict the map $f$ associated to $Z'$ 
to the set $\Scal_Z$.

\begin{prop}  \label{Open}
If $Z$ defines a triangle, then $f(\Scal_Z(X, \mu))$ is  open in  $T$.
\end{prop} 

\begin{proof}
Let $x$ belong to the interior of the triangle $P$ defined by $Z$.
Let $M$ be the maximal star convex neighborhood of $x$, and 
let $D \subset \R^2$ be the set $\dev_{\mu}(\overline{M}_x \cap \partial X)$. 
Since $\partial X$ is discrete, the set $D$ is discrete. It follows that the set 
\[  W~ =~ \left\{ q \in Q_{\widehat{\dev_{\mu}(Z)}}~ 
\left|~ q \left( \hat{D} \setminus \widehat{\dev_{\mu}(Z)} \right) 
\subset (0, \infty) \right. \right\} \]
is open in $Q_{\widehat{\dev_{\mu}(Z)}}$. By Proposition \ref{EllipseOpenFrontier} 
and Lemma \ref{MissingNothing}, 
\[  W'~ =~ \left\{ q \in Q_{\widehat{\dev_{\mu}(Z)}}~ |~ U_q 
           \mbox{ is an ellipse interior } \right\} \]
is open in $Q_{\widehat{\dev_{\mu}(Z)}}$. Thus, it suffices to show that
$f(\Scal_Z(X, \mu)) =   W \cap W' \cap T$.

Suppose that $U \in \Scal_Z(X, \mu)$.
Since $\overline{U} \cap \partial X=Z$, the set $\dev_{\mu}(\overline{U})$ 
does not intersect $D \setminus \dev_{\mu}(Z)$. It follows that 
$f(U) \in W$. Since $Z$ contains only three points, it follows from 
Proposition \ref{StripInfinite} that $f(U) \in W'$.
In sum,  $f(\Scal_Z(X, \mu)) \subset W \cap W' \cap T$. 

Let $q \in W \cap W' \cap T$. Since  $q \in W'$, the set $\overline{U_q}$ is convex,
and therefore contains $P$. It follows that $U_q$  contains $\dev_{\mu}(x)$. 
In other words,  $q\left(\widehat{\dev_{\mu}(x)} \right)<0$.  
Since $q \in W$ and $W \subset Q_{\widehat{\dev_{\mu}(Z)}}$, 
we have $q(\widehat{z}) \geq 0$ for each $z \in D$. 

Define $\gamma_z: [0, \infty) \rightarrow \R$ by 
\[   \gamma_z(t)~ =~ 
  q \left( (1-t) \cdot \widehat{\dev_{\mu}(x)}~ +~  t \cdot \widehat{z} \right).  \] 
Since the signature of $\underline{q}$ is $(2,0)$, the function $\gamma_z$   
has a unique global minimum $t_0$ and is strictly increasing for $t>t_0$. 
Thus, since $q\left(\widehat{\dev_{\mu}(x)} \right)<0$ and $q(\hat{z})\geq 0$, 
we have $\gamma_z(t) \geq  0$ for each $t \geq  1$. 

We have 
\[   \dev_{\mu}(M)~ =~ \R^2 \setminus  \bigcup_{z \in D} \gamma_z([1, \infty)).
\] 
Hence $U_q \subset \dev_{\mu}(M)$. Thus, $q=f(U)$ for $U=\dev_{\mu}|_{M}^{-1}(U_q)$. 
Therefore $ W \cap W' \cap T \subset f(\Scal_Z(X, \mu))$.
\end{proof}

\begin{prop} \label{fConvex}
If $Z \subset \partial X$ 
defines a polygon, then $f\left(\Scal_Z(X, \mu)\right)$ is convex. 
\end{prop}

\begin{proof}
Let $U, U' \in \Scal_Z(X, \mu)$.  
Since $U$ (resp. $U'$) is convex, $U$ (resp. $U'$) contains the 
interior of the polygon 
defined by $Z$. In particular,  $U\cap U' \neq \emptyset$.
Therefore, the union $U\cup U'$ is star 
convex, and hence the restriction of $\dev_{\mu}$ to $U\cup U'$ 
is injective. 

For $t \in [0, 1]$, consider $q_t= t \cdot f(U') + (1-t) \cdot f(U)$.
If $q_t(x,y,1)<0$, then either $q(x,y,1)< 0$ or $q'(x,y,1)< 0$.
Thus, for each $t \in [0,1]$, we have 
$U_{q_t} \subset \dev_{\mu}(U)\cup \dev_{\mu}(U')$,
and hence  $\dev_{\mu}|^{-1}_{U \cup U'}(U_{q_t})$ is a subconic 
in $\Scal_Z(X, \mu)$.
\end{proof}

\begin{prop} \label{fExtreme}
Let $Z$ define a polygon. The point $f(U)$ is an extreme point of  
$f\left(\overline{\Scal_Z(X, \mu)}\right)$ if and only if 
$U$ is a rigid subconic that belongs to $\overline{\Scal_Z(X, \mu)}$.
\end{prop}

\begin{proof}
Let $U$ be a rigid subconic in $\overline{\Scal}_Z$.
Suppose to the contrary that $f(U)$ is not an extreme point. That is, assume that
there exists $V$ and $V' \in \overline{\Scal}_Z$ such that 
$f(U)= \frac{1}{2} \left(f(V)+ f(V')\right)$.  Since $\overline{\Scal}_Z$
is convex and the set, $\Scal_5$, 
of rigid subconics is discrete, we may assume that neither 
$f(V)$ nor $f(V')$ is a rigid subconic. It follows that there 
exists $x \in \partial U \cap \partial X$ such that 
$x \notin \overline{V} \cup \overline{V'}$.
Thus, $f(V)(x) > 0$ and $f(V')(x) > 0$, and therefore 
$f(U)(x) >0$.  This contradicts the fact that $x \in \partial U \cap \partial X$.

Suppose that $U \in \overline{\Scal_Z(X, \mu)}$ is not a rigid subconic.
In particular, letting $Z'=\partial U \cap \partial X$, we have $\Card(Z')=3$ or $4$.
By Proposition \ref{ZPolygon}, the set $\hat{Z}'$ is in general position. 
Thus, by Proposition \ref{Dimension}, the vector space 
$Q_{\widehat{\dev_{\mu}(Z)}}$ has dimension $2$ or $3$.  
Thus, by Lemma \ref{MissingNothing}, there exists a nontrivial linear family 
$t \mapsto q_t$ of quadratic forms in $T$
such that $U_{q_0}= \dev_{\mu}(U)$ and $\dev_{\mu}(Z) \subset \partial U_{q_t}$. 
By Lemma \ref{Open}, there exists $\delta>0$ such that if 
$|t|< \delta$, then $U_{q_t} \in \dev_{\mu}(\Scal_Z(X, \mu))$. 
Therefore, $f(U)$ is not an extreme point. 
\end{proof}

The following theorem summarizes the preceding material and provides
the basis for the cell complex. In particular, if $Z$ 
defines a triangle and $\Scal_Z \neq \emptyset$,
then $\Scal_Z$ is homeomorphic to a cell.

\begin{thm} \label{CellTheorem}
Let $Z \subset \partial X$ define a triangle. 
The map $f$ is a homeomorphism from $\overline{\Scal_Z(X, \mu)}$ 
onto a compact convex planar polygon with finitely many sides.
Each side of $f(\overline{\Scal_Z(X, \mu)})$ equals 
$f(\overline{\Scal_{Z'}(X, \mu)})$ where 
$Z' \subset \partial X$ defines a quadrilateral and $Z \subset Z' \subset \partial X$. 
Each vertex of $f(\overline{\Scal_Z(X, \mu)})$ equals $f(U)$
where $U$ is a rigid conic with $Z \subset \partial U \cap \partial X$. 
A vertex $f(U)$ belongs to a side $f(\overline{\Scal_{Z'}(X, \mu)})$ 
if and only if $Z' \subset \partial U \cap \partial X$.
\end{thm}

\begin{proof}
By Proposition \ref{fConvex}, the 
image $f(\overline{\Scal}_Z)$ is convex.
By Proposition \ref{ScalCompactness} the set 
$f(\overline{\Scal}_Z)$ is compact and hence closed.
Therefore, by the Krein-Milman theorem and Proposition
\ref{fExtreme}, the set $f(\overline{\Scal}_Z)$
is the convex hull of the $f$-images of the rigid subconics
that belong to $\overline{\Scal}_Z$.
Each such rigid subconic contains the nonempty interior of the 
polygon $P_Z$. Therefore, by Corollary \ref{FiniteEllipses},
the set of such subconics is finite. In sum,  
$f(\overline{\Scal}_Z)$ is the convex hull of finitely 
many points in the 2-dimensional plane $T$. 

Since $f$ is a homeomorphism onto its image, 
Proposition \ref{Open} implies that $f(\Scal_Z)$
is the interior of the polygon $f(\overline{\Scal}_Z)$.
In particular, the boundary $\partial \Scal_Z$ consist of those 
subconics $U$ with $\Card(\partial U \cap \partial X) \geq 4$.
Since the rigid subconics correspond to extreme points, 
each side corresponds to a subset $Z' \subset X$ 
with $\Card(\partial U \cap \partial X)=4$. 
\end{proof}

\begin{coro} \label{CellComplex}
The space $\Scal_3(X, \mu)$ and the collection of cells
\[ \{\Scal_Z(X, \mu)~ |~ \Scal_Z(X, \mu) \neq \emptyset \mbox{ and } \Card(Z) \geq 3 \} \]
constitute a $2$-dimensional  cell complex.  
\end{coro}

\begin{remk}
The cell complex  $\Scal_3(X, \mu)$ is not a $CW$-complex.
Indeed, the topology on $\Scal_3(X, \mu)$ does not
give the weak topology with respect to the cells. 
If $U$ is a rigid strip,
then there exists a sequence of $1$-cells $\Scal_{Z_n}$
and points $U_n \in \Scal_{Z_n}$ such that $U_n$ 
converges $U$. The set $\{U_n\}$ is closed in the weak topology
but not in the natural topology chosen here.
On the other hand,  
Theorem \ref{FiniteEllipses} implies that 
the cell complex $\Scal_3(X, \mu)$ 
has the closure finiteness property.
\end{remk}

\begin{prop} \label{FourConnected}
The set $\Scal_4(X, \mu)$ is path connected.
\end{prop}

\begin{proof}
Let $U, V \in \Scal_4(X, \mu)$.  By Corollary 
there exists a path $\gamma:[0,1] \rightarrow \Scal_3(X, \mu)$ 
with $\gamma(0)=U$ and $\gamma(1)=V$.  Since each $2$-cell 
is bounded by finitely many $1$-cells in $\Scal_{\mu}(X, \mu)$
the path $\gamma$ can be homotoped to lie entirely in $\Scal_4(X, \mu)$.
\end{proof}


\section{Realizability and  adjacency} \label{SectionRealizable}

In this section $(X, \mu)$ is the universal cover of a precompact
translation surface with finite and nonempty frontier.

The cells of $\Scal_3(X,\mu)$ are indexed by the subsets $Z \subset \partial X$
such that $\Scal_Z(X, \mu) \neq \emptyset$ and $\Card(Z) \geq 3$. We will call 
such a set $Z$ a {\em realizable} set. 
By Theorem \ref{CellTheorem}, the inclusion of 
closed cells corresponds to the inclusion of realizable 
subsets of $\partial X$.   
In this section, we define the notion of `adjacency' 
and use it to characterize 
the subsets of a realizable set that are realizable.

\begin{defn}
Let $U$ be a subconic.  
We say that $x, y \in \partial U \cap \partial X$ are {\em adjacent} in 
$\partial U \cap \partial X$ if and only if there exists a connected component 
of $\partial U \setminus \{x,y\}$ that does not intersect $\partial X$. 
\end{defn}

\begin{prop}
Let $U$ be a subconic with ${\rm Card}(\partial U \cap \partial X) \geq 3$.
Each $x \in \partial U \cap \partial X$ is adjacent to exactly two points
in  $\partial U \cap \partial X$.
\end{prop}

\begin{proof}
Let $C$ be a connected component of $\partial U$. 
Let $\alpha: \R \rightarrow C$
denote a universal covering with $\alpha(0)=x$.
Since $\partial U \cap \partial X$ is discrete and $\alpha$ is a 
covering, $\alpha^{-1}(\partial X)$ is discrete.
If $U$ is a strip, then $\alpha$ is a homeomorphism.
It follows from Corollary \ref{StripInfinite}, both
$A_-= \{ s<0~ |~ \alpha(s) \in \partial X \}$ 
and $A_+ =\{ s<0~ |~ \alpha(s) \in \partial X \}$ are nonempty.
If $U$ is an ellipse interior, then $\pi_1(\partial U)\cong {\mathbb Z}$,
and it follows that  $A_+$ and $A_-$ are nonempty.
Let $x_{-}=\alpha(\sup(A_-))$ and $x_{+}=\alpha(\inf(A_+))$.
Then $x_{\pm}$ is adjacent to $x$, and since 
${\rm Card}(\partial U \cap \partial X) \geq 3$, we have $x_+ \neq x_-$.
\end{proof}

Suppose that $Z$ is realizable with $\Card(Z) \geq 5$.
If $Z' \subset Z$ is also realizable, then it follows from 
Proposition \ref{Dimension} that $\Card(Z')=3$ or $4$.

\begin{prop} \label{QuadrupleRealizable}
Let $Z \subset \partial X$ be realizable and let $Z' \subset Z$ 
with $\Card(Z')=4$.  The set $Z'$ is realizable 
if and only if there exists $U \in \Scal_Z$ such that
$Z'$ intersects each component of $\partial U$  and there exists a partition 
of $Z'$ into pairs $\{x_-,y_-\}$, $\{x_+,y_+\}$ such that 
each pair $\{x_{\pm}, y_{\pm}\}$ is adjacent in $\partial U \cap \partial X$.  
\end{prop}

\begin{proof}
We first note that the claim is true if $Z=Z'$. 
Indeed, if $Z=Z'$, then ${\rm Card}(Z) = 4$, and hence each subconic 
$U \in \Scal_Z$ is an ellipse interior. Thus, $\partial U$ is homeomorphic
to the unit circle,  and each point in $Z'= \partial X \cap \partial U$ is 
adjacent to exactly two other points in $Z'$. In particular,
there exists a partition of $Z'$ into adjacent pairs. 
Conversely, if $Z=Z'$, then since $Z$ is realizable,  $Z'$ is realizable.

Thus, we may assume that $Z' \neq Z$.

($\Rightarrow$)
If $Z'$ is realizable, then there exists $U' \in \Scal_{Z'}$. 
Since $Z \neq Z'$, we have  $U \neq U'$. Let $q \in Q(\R^3)$ 
(resp. $q' \in Q(\R^3)$) such that $U_q=\dev_{\mu}(U)$
and $U_{q'}=\dev_{\mu}(U')$. Since $U \neq U'$, the forms
$q$ and $q'$ do not belong to the same line. 
Let $F= \dev_{\mu}(Z)$.
Consider the components of the complement  $\partial U_q \setminus F$.
Since $\partial U$ is a 1-manifold, the intersection $\overline{C} \cap \overline{C'}$
of the closures of two components, $C$, $C'$, is either empty or is a singleton.  
In the latter case, we will say that the components are `adjacent'. 

By Proposition \ref{Dimension}, the zero locus of the restriction 
$q'|_{\partial U_q}$ equals $F$. (Indeed, if it were not the case, then 
$q$ and $q'$ would both vanish at 5 common points.)
Thus, since $q'|_{\partial U_q}$ is continuous, 
the restriction to each component $q'|_{C}$ is either positive or negative. 
We claim that if $C$ and $C'$ are adjacent, then 
the restrictions $q'|_{C}$ and $q'|_{C'}$ have opposite signs.  

Indeed, suppose to the contrary that the signs of $q'|_{C}$ 
and $q'|_{C'}$ are the same. Let $\alpha:(-\epsilon, \epsilon) \rightarrow \partial C_q$
be a differentiable path with $\alpha(0)$ equal to the point in
$\overline{C} \cap \overline{C'}$ and $|\alpha'(0)| \neq 0$. 
Note that $t \mapsto q \circ \hat{\alpha}(t)$ is constant,
and by   assumption $t \mapsto q' \circ \hat{\alpha}(t)$
has a critical point at $t=0$.  (Recall that $\widehat{(v_1,v_2)}=(v_1,v_2,1)$.) 
If we let $v_1= \hat{\alpha}(0)$ and let $w_1= \hat{\alpha}'(0)$,
then $dq_{v_1}(w_1)=0$ and $d(q')_{v_1}(w_1)=0$. Thus, by Proposition 
\ref{Dimension2}, the forms $q$ and $q'$ belong to the same line. 
This is a contradiction.

Since $Z'$ is realizable, $Z'$ intersects each component of $\partial U$.
We now identify the desired points $\{x_{\pm}, y_{\pm}\}$.

If $U$ is a strip, then let $\ell_{+}$, $\ell_-$ denote the two components.
Since $U \in S_{Z'}$ and ${\rm Card}(Z)=4$, we have ${\rm Card}(Z'\cap \ell_{\pm})=2$.
Indeed, otherwise $U'$ would contain three collinear points, and $U'$ 
would be a strip. But if $U'$ were a strip, then by Proposition \ref{StripInfinite},
the set $Z'= \partial U' \cap \partial X$ would be infinite. 
Since $\ell_{\pm}$ is homeomorphic to a line, the complement
$\ell_{\pm}/ \setminus Z$ has exactly two unbounded 
components and one bounded component. The frontier of each unbounded
component $C$ is a singleton and hence $U \cap C = \emptyset$.
Thus, $q$ is positive on the corresponding ray. Therefore, 
the form $q'$ is negative on the segment 
that corresponds to the bounded component, $B_{\pm} \subset \ell_{\pm}$. 
Thus $U' \supset B_{\pm}$, and we define $\{x_{\pm}, y_{\pm}\}$ to be the 
frontier points of $B_{\pm}$.

If $U$ is an ellipse interior, then $\partial U$ is homeomorphic to the
unit circle. Thus, $\partial U \setminus Z$ has four components that
we index with $\Z/ 4 \Z$ so that $C_i$ and $C_{i+1}$ are adjacent
for each $i \in \Z/4\Z$. For some $i \in \Z/4Z$, we have that 
the restriction of $q'$ to arc corresponding to $C_i$ is negative,  
and hence the restriction of $q'$ to the arc corresponding to 
$C_{i+2}$ is negative. In particular, $U' \supset C_{i} \cup C_{i+2}$.
Set  $\{x_{-}, y_{-}\}$ to be the 
frontier points of $C_i$, and set  $\{x_{+}, y_{+}\}$ equal to the 
frontier points of $C_{i+2}$.

($\Leftarrow$) Suppose that there exists $U \in S_Z$ such that 
$Z'$ admits a partition into two pairs of points adjacent in $Z$
and such that $Z'$ intersects each component of $\partial U$. 
Let  $F= \dev_{\mu}(Z)$ and $F'= \dev_{\mu}(Z')$. Let $q \in Q_{F}$
so that $\dev_{\mu}(U)=U_q$.

Since $Z'$ intersects each component of $\partial U$, the set
$Z'$ defines a polygon $P$. Since $F'=\dev_{\mu}(Z')$ is the set of extreme 
points of $K= \dev_{\mu}(P)$, the set $K$ is the intersection of four closed
half planes. Thus $\R^2 \setminus K$ is the union of four open half planes. 
Since $F' \neq F$ and $F \subset \R^2 \setminus K$, one of these
open  half planes, $H_1$, intersects $F$.  Let $\sigma_1$ be the side 
of $K$ that is contained in the closure of $\overline{H}_1$.  
Let $\sigma_3$ be the side of $P$
opposite to $\sigma$, and let $H_3$ be the open half plane whose 
closure $\overline{H}$ contains $\sigma_3$. Let $H_2$ and $H_4$
be the other two half-planes.

Let $p_-$ and $p_+$ be the endpoints of $\sigma_1$. Then $p_-, p_+ \in F'$. 
We claim that the corresponding points $x_-=\dev_{\mu}|^{-1}_{\overline{U}}(p_-)$ and 
$x_+=\dev_{\mu}|^{-1}_{\overline{U}}(p_+)$ are not 
adjacent in $Z=\partial U \cap \partial X$.

If $U_q$ is a strip and $x_-$ and $x_+$ were adjacent, then 
$p_-$ and $p_+$ would belong to the 
same component $\ell$ of $\partial U_q$. The component $\ell$ would 
coincide with the boundary of $H_1$, and $\partial U_q$ would
then intersect both components of $\R^2 \setminus \ell$.  Since 
$U_q$ is a strip, $\ell$ could not be the boundary of $U_q$.
Thus, in this case, $x_-$ and $x_+$ are not adjacent.

If $U_q$ is an ellipse interior, then the intersection $\partial H_i \cap \partial U_q$ 
consists of two points. Since $H_1 \cap \partial U_q \neq \emptyset$, the 
component $C$ of $\partial U_q \setminus F$ that joins 
$p_-, p_+ \in \partial H_1 \cap \partial U_q$
belongs to $H_1$. But $C \cap F \neq \emptyset$ and so $x_-$ and $x_+$ are not 
adjacent in this case.

Since $Z'$ has a partition into two adjacent pairs, $x_+$ (resp. $x_-$) is adjacent 
to an endpoint $y_+$ (resp. $y_-$) of $\sigma_3$. Let $C^+$ and $C^-$
denote the arcs with endpoints $\{x_{\pm}, y_{\pm}\}$ such that
$\partial X \cap C^{\pm}= \emptyset$. 

Let $\eta_i$ be a linear 1-form 
whose kernel contains the span of $\widehat{\partial H_i}$ and such that 
$\eta_i$ is negative on $\widehat{K^{\circ}}$ where $K^{\circ}$ is the interior of $K$.  
Then the quadratic form $q'= -\eta_1 \cdot \eta_3$ belongs to $Q_{F'}$.
Note that $H_1 \cup H_3= (q')^{-1}((0, \infty))$, and $q'$ is negative on $C^{\pm}$.

For $t \in \R$, define 
\[   q_t~ =~ t \cdot q'~ +~ (1-t) \cdot q. \]
If $t \in [0,1]$, then $U_{q_t} \subset U_q \cap U_{q'}$. 
Indeed, if $q_t(x)<0$, then either $q(x)<0$ or $q'(x)<0$.

Let $x \in K^{\circ}$, and let $M\subset X$ be the maximal star convex neighborhood of $x$.
Since $U$ is convex and $U_q= \dev_{\mu}(U)$, we have $U_q \subset \dev_{\mu}(M)$.

Let $S=\dev_{\mu}|_{M}^{-1}(U_{q'})$ and let $D= \dev_{\mu}(\overline{S} \cap \partial X)$. 
Since $\partial X \cap C^{\pm}= \emptyset$, we have $D \cap \dev_{\mu}(C^{\pm})= \emptyset$.
Since $\partial X$ is discrete, the set $D$ is discrete. Since $C^{\pm} \cap D = \emptyset$,
we have $q_{0}^{-1}((-\infty,0]) \cap D = F'$. Therefore since 
$t \mapsto q_t$ is continuous and $D$ is discrete, there exists $\epsilon \in (0,1)$
such that if $|t|< \epsilon$, then $q_{t}^{-1}((-\infty,0]) \cap D = F'$. 

We claim that there exists $\epsilon'>0$ such that if $t \in [0, \epsilon')$,
then $U_{q_t}$ is an ellipse interior. If $U_q$ is an ellipse, then 
this follows from Proposition
\ref{EllipseConvex} and the continuity of $t \mapsto q_t$.

Suppose that $U_q$ is a strip. Then $q|_{U_q}$ is bounded from below.
Indeed, $q= \eta_+ \cdot \eta_-$ where $\eta_-$ are linear 1-forms. Since 
kernel, $\ell_{\pm}$, of $\eta_{\pm}$ is parallel to the kernel, $\ell_{\mp}$,
of $\ell_{\pm}$, it follows that the form $\eta_{\pm}|_{U_q}$ is bounded. 
Hence $q|_{U_q}> -N$ for some $N>0$. 
Let $v$ belong to the kernel of $\eta_{\pm}$. For each $p \in \R^2$ define 
$f_p(s)=   q'( \hat{p}+s \cdot v)$. Note that $f_p$ is a quadratic polynomial
and we claim that $f_p$ is nontrivial with leading coefficient positive.  
Since ${\rm Card}(U_{q'} \cap U_q)=4$, each line in $q^{-1}\{0\}$ 
intersects $(q')^{-1}\{0\}$. It follows that $\eta_i(v) \neq 0$ for $i=1$ or $3$,
and hence $f_p(s)= (\eta_1(p) + s \cdot \eta_1(v)) \cdot (\eta_3(p) + s \cdot \eta_3(v))$
is nonconstant. Since $q'|_{\sigma_i}\equiv 0$ and $q'|_{K^0} \leq 0$, 
the leading coefficient of the quadratic polynomial 
$f_p$ is positive for each $p \in \sigma_1$.  
Since $\sigma_1$ is compact for each $t$, there exists $N'_t$ such that if $|s|>N'_t$
and $p \in \sigma_1$, then $q_t(p+s \cdot v)> (N+1)/t$. 
It follows that $U_{q_t} \cap U_q$ is bounded for each $t>0$. 
Since $U_q \cap U_{q'}$ is bounded, there exists $\epsilon' \in (0,1)$
so that if $|t|< \epsilon$, then  $U_{q_t} \cap U_{q'}$ is bounded.
Thus, for $t \in (0, \epsilon')$, we have that $U_{q_t}$ is bounded,
and hence by Proposition \ref{EllipseBounded}, the subconic 
$U_q$ is an ellipse interior. 

Fix $t \in (0,\min\{\epsilon, \epsilon'\})$. 
We have $U_{q_t} \setminus U_{q'} \subset U_q \subset \dev_{\mu}(M)$.
Thus, it suffices to show that  $U_{q_t} \cap U_{q'} \subset \dev_{\mu}(M)$.
For then $\dev|_{M}^{-1}(U_{q_{t_0}}) \in \Scal_{Z'}$.  

If $y \in U_{q_t} \cap U_q$, then $y \notin D$. 
Since $U_{q_t}$ is convex, the line segment $\sigma$ joining $y$ and $\dev_{\mu}(x)$
belongs to $U_{q_t}$. Therefore, $D \cap \sigma = \emptyset$, and it
follows that $y \in \dev_{\mu}(M)$. 
\end{proof}

\begin{prop}  \label{TripleRealizable}
Let $Z \subset \partial X$ be realizable and $Z' \subset Z$ with $\Card(Z')=3$.
The triple $Z'$ is realizable 
if and only if there exists $U \in \Scal_Z$ and 
a pair $\{x,x'\} \subset Z'$ that is adjacent in $\partial U \cap \partial X$.  
\end{prop}

\begin{proof}
The proof is similar to the proof of Proposition \ref{QuadrupleRealizable}.
We leave it to the reader. 
\end{proof}


\section{Orientation and succession}  \label{SectionOrientation}

In this section, we first observe that 
each $2$-cell of $\Scal(X, \mu)$ has a canonical orientation.
We then reinterpret this orientation in terms of {\em succession},
the natural refinement of adjacency. We will assume 
that $(X, \mu)$ is the universal covering of a precompact 
translation surface with nonempty and finite frontier.

Let $Z \subset X$ be a realizable triple, and thus, 
in particular, $\dev_{\mu}(Z)$ is noncollinear. Let $\vec{d}$ is an 
oriented negative natural basis for $Q_{\widehat{\dev_{\mu}(Z)}}$,
and let $r_{\vec{d}}: T_{\vec{d}} \rightarrow \Scal_{\dev_{\mu}(Z)}$ be 
the homeomorphism defined in \S \ref{SectionZ}. Recall that the plane  
$T_{\vec{d}}$ has a canonical (outward normal) orientation, and hence 
the homeomorphism $r_{\vec{d}}$ induces an orientation on $\Scal_Z$. 
By Proposition \ref{RestrictionNatural},
this orientation does not depend on the choice of oriented
negative natural basis. As a result, the cell $\Scal_Z$ has a 
canonical orientation. 

The orientation of $\Scal_Z$ induces an orientation of its boundary
$\partial \Scal_Z$. In particular, it induces a cyclic ordering of the 
$1$-cells lying in the boundary of $\Scal_Z$. These $1$-cells are in 
one-to-one correspondence with the set, $\partial Z$, of 
realizable quadruples $Z'$ that contain $Z$. 

\begin{defn}
Let $Z', Z'' \in \partial Z$. We will say $Z'$
{\em follows} $Z''$ if and only if $\Scal_{Z'}$ immediately
follows $\Scal_{Z''}$ in the 
canonical ordering of $\partial Z$.  
\end{defn} 

The cyclic ordering has an alternate description that uses the oriented
refinement of adjacency. In particular, if $Z''$ follows $Z'$, then 
$Z''$ and $Z'$ share a vertex and this vertex corresponds to the unique
rigid conic $U$ such that $Z'' \cup Z' \subset \partial U \cap \partial X$.
In Proposition \ref{1-CellsFollow} below, we reinterpret `following'
in terms of an ordering of $\partial U \cap \partial X$.

\begin{defn}
Suppose that ${\rm Card}(\partial U \cap \partial X) \geq 3$.
Given adjacent points $x, y \in \partial U \cap \partial X$,
let $C$ be the unique component of $\partial U \setminus \partial X$
such that $\{x,y\} \subset \overline{C}$. 
If there exists an oriented path $\alpha: [-1,1] \rightarrow \overline{C}$
such that  $\alpha(-1)=x$ and $\alpha(+1)=y$, then we say that
$y$ is the {\em successor} of $x$ in $\partial U \cap \partial X$, 
and we write $s_U(x)=y$.
\end{defn}

\begin{prop}
The map $s_U$ is a permutation of $\partial U \cap \partial X$.
\end{prop}

\begin{proof}
Straightforward.
\end{proof}

In the remainder of this section, we will assume that $U$ is a rigid subconic.
That is, we assume that ${\rm Card}(\partial U \cap \partial X) \geq 5$. 

\begin{nota}
Given $x, x' \in \partial U \cap \partial X$,
set 
\[    Z_U(x,x')~ :=~ \{x, s_U(x), x', s_{U}(x') \}. \]
\end{nota}

If $x \neq x'$ and $\{x, x'\}$ are nonadjacent, then $Z(x,x')$
is a quadruple, and hence by Proposition \ref{QuadrupleRealizable},
this quadruple is realizable. Conversely, if $Z$ is realizable 
and $U \in \overline{\Scal_{Z}}$, then $Z$ can be partitioned
into adjacent pairs $\{x_-,y_-\}$ and $\{x_+,y_+\}$. For each pair 
either $s_U(x_{\pm})=y_{\pm}$ or $s_U(y_{\pm})=x_{\pm}$, but not both. 
Thus, $(x,y) \mapsto Z_U(x,y)$ is a bijection from the 
set of nonadjacent pairs in $\partial U \cap \partial X$
onto the realizable quadruples $Z \subset \partial U \cap \partial X$.

\begin{defn} \label{ConsecutiveDefinition}
We will say that $Z \subset \partial U \cap \partial X$ 
is {\em consecutive} in $\partial U \cap \partial X$ if and only 
if there exists $x \in Z$ 
and $k \in {\mathbb Z}^+$ such that $Z= \{x, s_U(x), \ldots, s_U^k(x)\}$.
Let $x \in Z$.  If for all $y \in Z$, we have $s_U(y) \neq x$ and 
$s_U(x) \neq y$, then we will say that $x$ is 
{\em $Z$-isolated in} $\partial U \cap\partial X$. 
\end{defn}

Let $Z \subset \partial U \cap \partial X$ be a realizable triple.
Then either $Z$ is consecutive, namely, $Z=\{x, s^2_U(x), s^{3}(x)\}$
for a unique $x \in Z$, or  there exists  $y \in Z$  
that is isolated, that is $Z= \{x, s_U(x), y\}$ for a unique $x \in Z$.

\begin{prop} \label{1-CellsFollow}
Let $U$ be a rigid subconic and let $Z= \partial U \cap \partial X$ be a realizable triple.
If $Z= \{x, s_U(x), s^2_U(x)\}$ is consecutive in $\partial U \cap \partial X$, then 
$Z_U(x,s^{2}_U(x))$ follows $Z_U(s^{-1}_U(x),s_U(x))$ in $\partial Z$. 
If $Z= \{x, s_U(x), y\}$ is nonconsecutive, then $Z_U(x,y)$ follows 
$Z_U(x, s^{-1}(y))$ in $\partial Z$.
\end{prop}

\begin{proof}
Suppose that $Z= \{x, s_U(x), s^2_U(x)\}$ is consecutive in $\partial U \cap \partial X$.
Let $v_i= \dev_{\mu}(s_U^i(x))$ for $i \in \Z$.  We first note that  the ordered triple 
$(v_0,v_1,v_2)$ is cyclically ordered with respect to the standard orientation of $\R^2$.
Indeed, $Z$ belongs to a single component of $\partial U$, and thus, since $Z$
is realizable, $\partial U$ has only one component. In particular, $U$ is an ellipse
interior. It is then straightforward to construct a homotopy of piecewise smooth
embeddings $f_t: \R / \Z \rightarrow \R^2$ such that $f_0$ is an oriented
parametrization of $\dev_{\mu}(\partial U)$, $f_1$ is a parametrization of 
the boundary of the convex hull of $\{v_0, v_1, v_2\}$, 
and $f_t([i/3])= x_i$ for all $t \in [0,1]$ and $i=0,1,2$.  

Let $\vec{d}=(d_0,d_1,d_2)$ be an oriented negative natural basis 
associated to the cyclically ordered noncollinear triple $(v_0,v_1,v_2)$. 
In particular, for $\{i,j,k\}=\{1,2,3\}$, let $\eta_{ij}$ be the linear form with kernel 
$\langle \hat{v}_i, \hat{v}_j \rangle$ such that $\eta_{ij}(\hat{v}_k)=1$,
and set 
\[  d_i~  =~  -\eta_{i,k} \cdot \eta_{i,j}.\] 
Let $T= \{(t_0,t_1,t_2)~ |~ \sum t_i=1\}$,
let $r: T \rightarrow \Scal_{\{v_1,v_2,v_3\}}$ be the
homeomorphism $r(t_0,t_1,t_2)= \sum t_i \cdot d_i$, and 
let $f = r^{-1} \circ \dev_{\mu}$.  

For $i=0,1$, let $V_i= \{v_{i-1}, v_{i},v_{i+1}, v_{i+2}\}$,
and let $\ell_i$ denote the set
\[ r^{-1}(\Scal_{V_i})~  =~ Q_{\hat{V_i}} \cap T. \]
By Proposition \ref{Dimension}, the subspace $\hat{Q}_{V_i}$ is $2$-dimensional, 
and thus it follows that $\ell_i$ is an affine line. 
Since $\dev_{\mu}(Z(s^{-1}(x), s(x)))=V_0$ (resp. $\dev_{\mu}(Z(x, s^2(x)))=V_1$),
we have $f(\Scal_{Z(s^{-1}(x), s(x))}) \subset \ell_0$
(resp. $f(\dev_{\mu}(Z(x, s^2(x))) \subset \ell_1$).    

The $\ell_0 \cap \ell_1$ meet at the unique point $q \in T$ such that 
if $j=-1,\ldots, 3$, then $q(\hat{v}_{j})=0$. 
In particular, $\dev_{\mu}(U)=U_q$. Since, for example, 
$U$ is an extreme point, $\ell_0 \cap \ell_1 = \{q\}$.

Let $\ell'_i \subset T$ be the line containing $d_i$ and $d_{i+1}$.
The lines $\ell_i$ and $\ell_i'$ intersect at one point $a_i$.
It will be convenient to give an explicit construction of this point. 
If $i=0$, the construction runs as follows: Note that 
$\eta_{1,2}(\hat{v}_{-1})<0$ and $\eta_{0,2}(\hat{v}_{-1})>0$,
and hence $-\eta_{1,2}(\hat{v}_{-1})/ \eta_{0,2}(\hat{v}_{-1}) >0$.
Thus, there exists a unique $t_0 \in (0,1)$ such that 
\begin{equation}  \label{vminus1}
 \frac{t_0}{1-t_0}~ =~ - \frac{\eta_{1,2}(\hat{v}_{-1})}{\eta_{0,2}(\hat{v}_{-1})}.  
\end{equation}
Define 
\[   \eta_{-1,2}~ =~  t_0 \cdot \eta_{0,2}~ +~  (1-t_0)\cdot \eta_{1,2} \]
and 
\[ a_0~ =~ - \eta_{-1,2} \cdot \eta_{0,1}.\]
In other words,
\[   a_0~ =~ t_0 \cdot d_0~ +~ (1-t_0) \cdot d_1. \]
It follows from (\ref{vminus1}) that $\eta_{-1,2}(v_{-1})= 0$,
and hence $a_0 \in Q_{\hat{V_0}}$.

Note that $a_0(\hat{v}_3)>0$. Indeed, since $\{v_{-1}, v_0,v_1,v_2\}$ are consecutive
$\hat{v_0}$ and $\hat{v}_3$ lie in distinct components of 
$\R^3 \setminus \langle \hat{v}_{-1}, \hat{v}_{2} \rangle$. Thus, since
$\eta_{-1,2}(\hat{v}_0)=(1-t_0)>0$ we have $\eta_{-1,2}(\hat{v}_3)<0$.
Since $\hat{v}_2$ and $\hat{v}_3$ lie in the same component of
$\R^3 \setminus \langle \hat{v}_{-1}, \hat{v}_{2} \rangle$ and 
$\eta_{0,1}(\hat{v}_2)=1>0$, we have $\eta_{0,1}(\hat{v}_3)>0$.
Thus, 
\[ a_0(\hat{v}_3)~ =~ -\eta_{0,1}(\hat{v}_3) \cdot \eta_{-1,2}(\hat{v}_3)~ >~ 0. \]

A similar construction produces $t_1 \in (0,1)$ such that
\[  a_1~ =~ t_1 \cdot d_1~ +~ (1-t_1) \cdot d_2, \]
$a_1 \in \ell_1$, and $a_1(\hat{v}_{-1})~ >0$.

We claim that $a_i \neq q$ for $i=1,2$.
Since the set $\{x, s(x), s^2(x)\}$ is consecutive, this set is a subset of 
one boundary component of $\partial U$. Thus, since $\{x, s(x), s^2(x)\}$
is realizable, $U$ is an ellipse interior and $q$ is nondegenerate.
Therefore since $a_i$ is degenerate, $a_i \neq q$.

We claim that  $f(\Scal_{Z(x,s^2(x))})$ is contained in the ray $\overrightarrow{qa_0}$
and that $f(\Scal_{Z(s^{-1}(x),s(x))})$ is contains in the ray $\overrightarrow{qa_1}$.
Indeed, for $i=0,1$, let
\[  q_t^i~ =~  (1-t) \cdot q~ +~ t \cdot a_i, \]
Since $a_i \neq q$, the affine function $t \mapsto q_t$ maps $\R$ onto $\ell_i$. 
Since $a_0(\hat{v}_{3})>0$ (resp. $a_1(\hat{v}_{-1})~ >0$) and 
$q(\hat{v}_{3})=0$ (resp. $q(\hat{v}_{-1})=0$), the affine function 
$t \mapsto q^0_t(\hat{v}_{3})$ (resp. $t \mapsto q^1_t(\hat{v}_{-1})$) 
is increasing. Therefore, for $t <0$, we have $q^0( \hat{v}_{3})<0$ 
(resp. $q^1( \hat{v}_{-1})<0$) and hence $q_t \notin f(\Scal_{Z(s^{-1}(x),s(x))})$
(resp. $q_t \notin f(\Scal_{Z(x,s^2(x))})$).

Let $\ell$ be the line containing $a_0$ and $a_1$.
We claim that if $q' \in \ell$, then either $q'(\hat{v}_{-1})>0$ 
or  $q'(\hat{v}_{3})>0$. Let $a_t= (1-t) \cdot a_0 + t \cdot a_1$.
We have $a_0(\hat{v}_3)>0$ and $a_1(\hat{v}_{3})=0$.
Therefore, if $t>0$, then $a_t(v_{3})>0$.  We have 
$a_1(\hat{v}_{-1})>0$ and $a_0(\hat{v}_{-1})=0$.
Therefore, if $t<1$, then $a_t(\hat{v}_{-1})>0$.

Define $q_t= (1-t) \cdot q + t \cdot d_1$.  We claim that 
for all $t\geq0$, the quadratic form $q_t$ does not lie in $\ell$.
Recall that $d_1= - \eta_{0,1}\cdot \eta_{1,2}$. The points 
$\hat{v}_2$, $\hat{v}_3$, and $\hat{v}_{-1}$ lie in the same component 
of $\R^2 \setminus \langle \hat{v}_0, \hat{v}_1\rangle$, and hence 
$\eta_{0,1}(\hat{v}_3)>0$ and $\eta_{0,1}(\hat{v}_{-1})>0$.
The points $\hat{v}_0$, $\hat{v}_3$, and $\hat{v}_{-1}$ lie in the same component 
of $\R^2 \setminus \langle \hat{v}_1, \hat{v}_2\rangle$, and hence 
$\eta_{1,2}(\hat{v}_3)>0$ and $\eta_{1,2}(\hat{v}_{-1})>0$.
It follows that $d_1(\hat{v}_3)<0$ and  $d_1(\hat{v}_{-1})<0$.
Therefore, since $q( \hat{v}_3)=0= q(\hat{v}_{-1})$, if $t \geq 0$,
then  $q_t(\hat{v}_3) \leq 0$ and $q_t(\hat{v}_{-1}) \leq 0$.

By the choice of orientation of $T$, we have that $(d_2-d_1, d_0-d_1)$
is an oriented basis for $\R^2$. Since $t_0, t_1 \in (0,1)$, 
the ordered pair $(a_1-d_1,a_0-d_1)$ is an oriented basis for $\R^2$.
If $t \geq 0$, the quadratic form 
$q_t$ does not lie in the line $\ell$ containing $a_0$ and $a_1$.
Therefore, for $t \geq 0$, the set $\{a_1-q_t,a_0-q_t\}$ is a basis
and by continuity,  $(a_1-q,a_0-q)$ is an oriented basis. 
Thus, $f(\Scal_{Z(x,s^2(x))})$ follows $f(\Scal_{Z(s^{-1}(x),s(x))})$.

The proof of the other claim is similar. 
\end{proof}


\section{The link of a vertex}  \label{SectionLink}

Recall that if $v$ is a vertex in a $2$-dimensional cell complex, 
then the link of $v$, $\lk(v)$, is the abstract graph defined as follows:  
The vertices are the $1$-cells that contain $v$.
Two $1$-cells $C, C' \subset \partial U \cap \partial X$  
are joined by an edge iff there exists a $2$-cell  
$D$ such that $v \in \partial C \cap \partial C'$ and
$C \cup C' \subset \partial D$.

In this section, we study the link $\lk(U)$ of a rigid conic 
$U$ in $\Scal_3(X, \mu)$. 
The vertices of $\lk(U)$ may be regarded as realizable quadruples 
$Z \subset \partial U \cap  \partial X$.  Two quadruples $Z$ and $Z'$  
are joined by an edge if and only $Z \cap Z'$ is a realizable triple.
The orientation of the $2$-cell $\Scal_{Z \cap Z'}$ determines a direction 
of this edge.  Hence, $\lk(U)$ is naturally a directed graph. 

We will assume throughout this section 
that $(X, \mu)$ is the universal covering of a precompact 
translation surface with nonempty and finite frontier. 

By Proposition \ref{QuadrupleRealizable}, the set of $1$-cells that contain $U$
is in bijection with the set of unordered pairs $\{x,y\}$ 
of non-adjacent points in $\partial U \cap \partial X$.  
If $U$ is an ellipse, then the convex hull of $\partial U \cap \partial X$
is an $n$-gon with $n={\rm Card}(\partial U \cap \partial X)< \infty$.
The non-adjacent vertices define `diagonals' in the $n$-gon. 
The classic count of diagonals gives the following.

\begin{prop} \label{CardLink}
If $U$ is a rigid  ellipse and $\Card(\partial U \cap \partial X)=n$, then ${\rm Lk}(U)$
has $n(n-3)/2$ vertices. 
\end{prop}

\begin{figure}   
\begin{center}

\includegraphics[height=4in,width=4in]{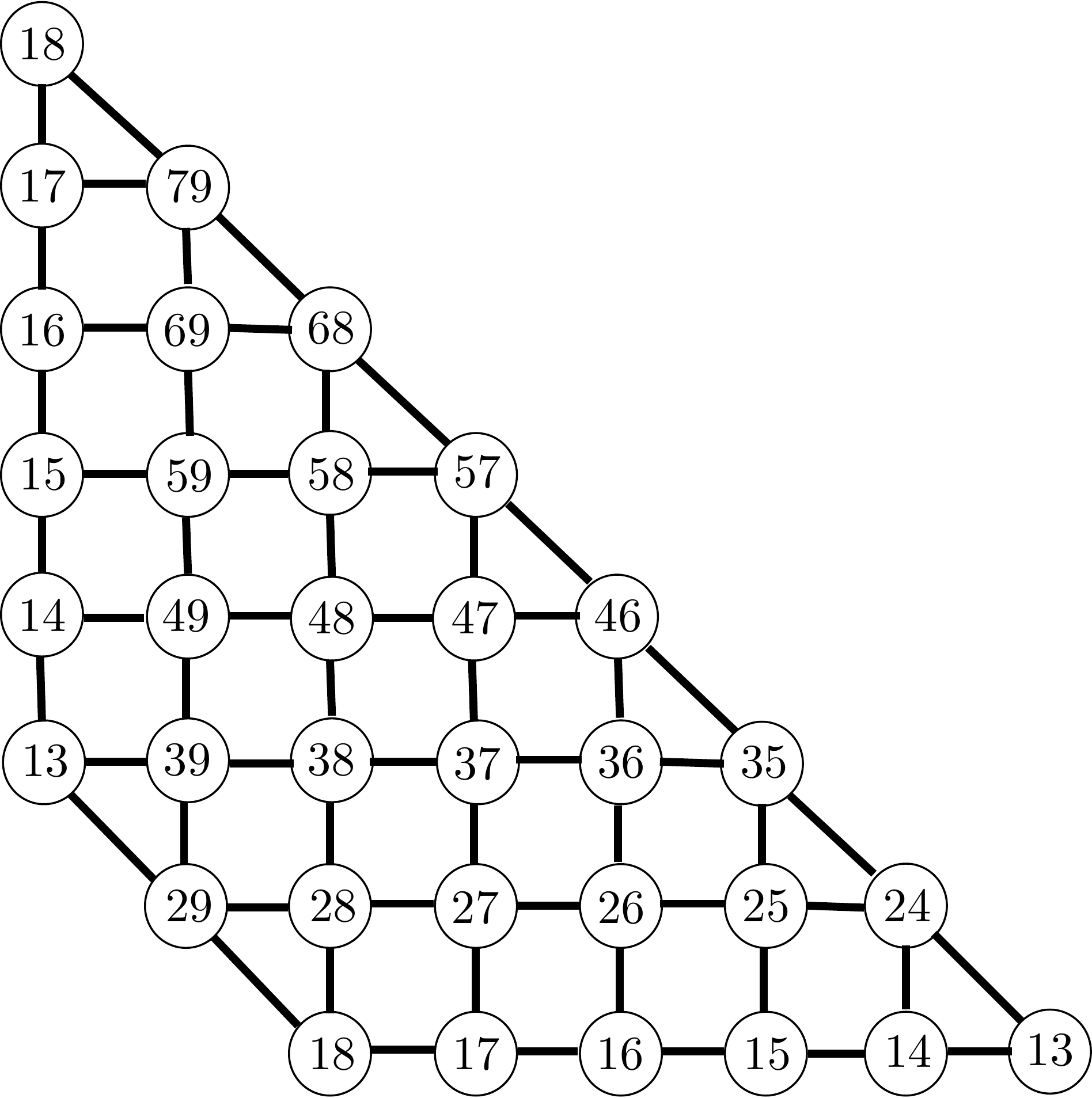}

\end{center}
\caption{The link of $U$ when $\Card(\partial U \cap \partial X)=9$.
The vertex corresponding to $Z_U(x_i,x_j)$ is labeled $ij$.
Note that the vertices labeled $13$, $14$, $15$, $16$, $17$, and $18$  appear twice. 
These pairs should be identified to obtain a graph that can be embedded
in the M\"obius band.}
\end{figure}

To simplify the notation in what follows, we will write $x+k$ in place
of $s_U^k(x)$.

\begin{prop} \label{Degree4}
Each vertex of  ${\rm Lk}(U)$ has degree 4.
If the quadruple $Z(x,y)$ is consecutive, then the neighbours of $Z(x,y)$ are 
\begin{equation} \label{Neighbours1}
Z(x-1,x+2),~  Z(x-1, x+1).
\end{equation}
If the quadruple $Z(x,y)$ is not consecutive, then the neighbours of $Z(x,y)$ 
are 
\begin{equation} \label{Non-consecutive}
 Z(x,y-1),~  Z(x, y+1),~  Z(x+1,y),~  Z(x-1, y).
\end{equation}
\end{prop}

\begin{proof}
A quadruple $Z'$ is a neighbour of $Z_U(x,y)$ if and only 
if the intersection $Z' \cap Z_U(x,y)$ consists 
of three points. 

If $|x-y|=2$, then either $Z_U(x,y)= \{x,x+1,x+2,x+3\}$ or $Z_U(x,y)= \{y,y+1,y+2,y+3\}$.
Without loss of generality, the former holds.
In this case, the possibile intersections 
are $\{x,x+1,x+2\}$,  $\{x+1,x+2, x+3\}$, $\{x, x+2, x+3\}$, and $\{x,x+1,x+3\}$.
By Proposition \ref{1-CellsFollow}, the only other quadruple that contain
these triples are listed in (\ref{Neighbours1}). 

If $|x-y|>2$, then  $Z_U(x,y)= \{x,x+1,y,y+1\}$.
The possible intersections are 
$\{x,x+1,y\}$,  $\{x,x+2, y+1\}$, $\{x, y, y+1\}$, and $\{x+1, y,y+1\}$.
By Proposition \ref{1-CellsFollow}, the only other quadruples that 
contains these are triples are listed in (\ref{Non-consecutive}).
\end{proof}

We next consider the $3$-cycles in the undirected graph $\lk(U)$.

\begin{lem} \label{3cycles}
The vertex $Z_U(x,x+2)$ belongs to exactly two  $3$-cycles:
\[ \{Z(x,x+2),Z(x-1,x+2), Z(x-1, x+1) \} \]
and 
\[   \{Z(x,x+2), Z(x,x+2),Z(x+1, x+3)\}.  \]
Suppose that $Z_U(x,y)$ is not consecutive with respect to $U$.
If a 3-cycle contains $Z_U(x,y)$, then either $x = y+3$ and 
the cycle is 
\[    \{Z(x,y),   Z(x,y+1), Z(x-1,y) \} \]
or $y=x+3$ and the cycle is 
\[      \{Z(x,y), Z(x+1,y), Z(x,y-1) \}\]
\end{lem}

\begin{proof}
Suppose that $Z=\{x,x+1,x+2,x+3\}$ is consecutive with respect to $U$. 
Then by Proposition \ref{Degree4}, the neighbours of $Z$ are
$Z(x-1,x+2)$, $Z(x-1, x+1)$,  $Z(x,x+3)$, and  $Z(x+1, x+3)$.
Inspection shows that the first two intersect  
in 3 points and the last two intersect in three points.
No other pairs intersect in three points. 
Thus, we have exactly two $3$-cycles containing $A$.

Suppose that $Z=Z(x,y)$ is not consecutive. Then by Proposition 
\ref{Degree4} the neighbours of $Z(x,y)$ are $Z(x+1,y)$, $Z(x,y+1)$,
$Z(x-1,y)$ and $Z(x, y-1)$. Since $Z(x,y)$ is not consecutive, we 
have $x \neq y-2, y-1,y, y+1, y+2$. Since $\Card(\partial U \cap \partial X)\geq 5$
we have $x \neq x-2, x-1, x, x+1, x+2$.  Inspection shows that
\begin{eqnarray*}
   Z(x+1,y) \cap Z(x, y+1)~ =~ \{x+1,y+1\}  \\
   Z(x+1,y) \cap Z(x-1, y)~ =~ \{y, y+1\}  \\
  Z(x,y+1) \cap Z(x, y-1)~ =~ \{x, x+1\}  \\
  Z(x-1,y) \cap Z(x, y-1)~ =~ \{x, y\}  
\end{eqnarray*}
If $y+2=x-1$, then $Z(x,y+1) \cap Z(x-1,y)= \{x, y+1, y+2\}$
and if $y+2 \neq x-1$, then $Z(x,y+1) \cap Z(x-1,y))= \{x, y+1\}$. 
If $x+2=y-1$, then $Z(x+1,y) \cap Z(x,y-1)= \{x+1, x+2, y\}$
and if $y+2 \neq x-1$, then $Z(x,y+1) \cap Z(x-1,y))= \{x+1,y\}$. 

In sum, if $y+2 \neq x-1$ and $y+2 \neq x-1$, then $Z(x,y)$
does not belong to a 3-cycle. If $y+2 \neq x-1$ or $y+2 \neq x-1$
but not both, then $Z(x,y)$ belongs to exactly one 3-cycle. 
Finally, $y+2 = x-1$ and $y+2 = x-1$ if and only if 
$Z(x,y)$ belongs to exactly two 3-cycles. 

In particular, if $Z(x,y)$ is nonconsecutive and belongs
to two 3-cycles, then $x=x+6$ and hence $\Card(\partial U \cap \partial X)=6$.    
\end{proof}

\begin{coro}
Each $3$-cycle contains at most one non-consecutive vertex.
\end{coro}

\begin{coro}
If $U$ is a maximal strip, then ${\rm Lk}(U)$ has no 3-cycles.
\end{coro}

\begin{proof}
Indeed, if a quadruple $Z$ is consecutive, then $Z$ is a subset of 
one component of the boundary of the strip. 
Hence $Z$ is collinear and is not realizable. 
\end{proof}

\begin{prop} \label{HexagonDirected}
Let $\{Z, Z', Z''\}$ be a $3$-cycle. Then $Z$ is consecutive 
if and only if either both $Z''$and $Z'$ follow $Z$
or $Z$ follows both $Z$ and $Z''$.  
\end{prop} 

\begin{proof}
If $Z= Z(x,x+2)$ is consecutive, then by Proposition \ref{3cycles},
the 3-cycles given by $\{Z(x,x+2),Z(x-1,x+2)$, $Z(x-1, x+1)\}$
and  $\{Z(x,x+2), Z(x,x+2),Z(x+1, x+3)\}$. In the former
case $Z(x,x+2)$ follows both $Z(x-1,x+2)$ and $Z(x-1, x+1)$.
In the latter, $Z(x,x+2)$ and $Z(x+1, x+3)$ both $Z(x,x+2)$.

If $Z=Z(x,y)$ is nonconsecutive and belongs to 
a 3-cycles, then by Proposition \ref{3cycles},
either $x=y+3$ and
\[    \{Z(x,y),   Z(x,y+1), Z(x-1,y) \} \]
or $y=x+3$ and the cycle is 
\[      \{Z(x,y), Z(x+1,y), Z(x,y-1) \}. \]
In the former case, $Z(x,y)$ follows $Z(x-1,y)$ and 
$Z(x,y+1)$ follows $Z(x,y)$. In the latter case,
$Z(x,y)$ follows $Z(x,y-1)$ and $Z(x+1,y)$ follows $Z(x,y)$. 
\end{proof}

Next we turn to the analysis of $\lk(U)$ in the case that $U$
is a maximal strip. First recall that the {\em $n$-dimensional infinite grid}, 
${\rm Grid}^n$, is the graph whose vertex set is $\Z^n$ and 
 two vertices $\vec{m}$ and $\vec{n}$ 
are joined by an edge iff $\vec{m}$ and $\vec{n}$
differ in exactly one coordinate and this difference 
has absolute value 1. 
This graph can be geometrically realized as the set 
of points in $\R^n$ all but one of whose coordinates
equals an integer. Any graph that is isomorphic to
the $n$-dimensional infinite grid 
will be called an {\em $n$-dimensional grid graph}.

\begin{prop} \label{GridGraph}
If $U \subset X$ is a maximal strip, then $\lk(U)$ is a $2$-dimensional 
grid graph. In particular, if $x$ and $y$ are points in the frontier of $X$
that belong to distinct components of $\partial U$, then 
\[   f(m, n)~ =~   Z(x+m,y+n) \]  
defines an isomorphism $f: {\rm Grid}^2 \rightarrow \lk(U)$. 
\end{prop}

\begin{proof}
Let $\ell_x$ (resp. $\ell_y$)  
denote the component of $\partial U$ containing $x$ (resp. $y$.)
For each $x' \in \ell_x$ (resp. $y' \in \ell_y$) 
there exists a unique $m$ (resp. $n$) such that $x'=x+m$ (resp. $y'=y+n$).
It follows that $f$ is a bijection. 
Since $U$ is a strip, each $1$-cell is non-consecutive.
Therefore (\ref{Non-consecutive}) implies
that $f$ is a graph isomorphism. 
\end{proof}

We will say that a subset $L$ of a grid graph $G$ is a {\em line}
if and only if $L$ is isomorphic to a 1-dimensional grid
and no two consecutive edges of $L$ belong to a cycle of order $4$. 
Let ${\mathcal L}(G)$ denote the set of lines in $G$.
The lines in the infinite grid correspond exactly to the 
lines in $\R^n$ that are contained within its geometric realization.


\section{The map of frontiers}  \label{SectionFrontier}

Let $(X, \mu)$ and $(X', \mu')$ be universal covers of precompact 
translation surfaces with $\partial X$ finite and nonempty. 
Let $\Phi: \Scal_3(X, \mu)  \rightarrow  \Scal_3(X', \mu')$ be an
isomorphism of 2-complexes. In this section, we will assume that 
$\Phi$ preserves the natural orientation of each $2$-cell.  
Such an isomorphism will be called  {\em orientation preserving}.

The goal of this section is to prove the existence of a bijection 
$\beta: \partial X \rightarrow \partial X'$ so that 
for each cell in the 2-complex $\Scal_3(X, \mu)$ we have 
\begin{equation}  \label{Diagram}
  \Phi(\Scal_Z)~ =~ \Scal_{\beta(Z)}. 
\end{equation}
We will reduce the construction of $\beta$ to a `local' problem 
associated to a rigid subconic $U$. In particular, we say that a bijection 
$\beta_U: \partial U \cap \partial X \rightarrow  \partial \Phi(U) \cap \partial X'$
is {\em adapted to $\Phi$} if and only if for each 
realizable $Z \subset \partial U \cap \partial X$ we have 
\begin{equation}  \label{Diagram2}
  \Phi(\Scal_Z)~ =~ \Scal_{\beta_U(Z)}. 
\end{equation}

The following proposition tells us that the `global' construction
of $\beta$ is equivalent to the `local' problem of constructing $\beta_U$
adapted to $\Phi$ for each rigid conic $U$.

\begin{prop}
Let $U$ and $U' \in \Scal_5(X, \mu)$.
If $\beta_U$ and $\beta_{U'}$ are adapted to $\Phi$, then 
for each $x \in \partial U \cap \partial U' \cap \partial X$ 
we have $\beta_U(x)= \beta_{U'}(x)$.  
\end{prop}

\begin{proof}
By Proposition \ref{FourConnected},
it suffices to suppose that $U$ and $U'$ are endpoints of a $1$-cell in $\Scal_3(X, \mu)$.
In particular, $Z= \partial U \cap \partial U'$
is a quadruple $\{x_1, x_2, x_3, x_4\}$.  Let $Z_i$ be the triple
$Z \setminus \{z_i\}$. The triple $Z_i$ is realizable, and hence 
\[  \Scal_{\beta_U(Z_i)}~ =~ \Phi(\Scal_{Z_i})~ =~ \Scal_{\beta_{U'}(Z_i)}.  \]
Hence, $\beta_U(Z_i)= \beta_{U'}(Z_i)$ for each $i$. 
Therefore, 
\begin{eqnarray*}
 \beta_U(\{x_j\})  &=& \beta_U \left(\bigcap_{i \neq j} Z_i\right)  \\
          &=&      \bigcap_{i\neq j} \beta_U(Z_i)  \\
        &=&      \bigcap_{i\neq j} \beta_{U'}(Z_i)  \\
             &=&  \beta_{U'}\left(\bigcap_{i\neq j} Z_i \right)  \\
       &=&   \beta_{U'}(\{x_j\}).
\end{eqnarray*}
\end{proof}

\begin{nota}
Since each cell is determined by a subset of the frontier, the 
map $\Phi$ may be regarded as a map from certain subsets 
of $\partial X$ to subsets of $\partial X'$. Abusing notation
slightly, we will sometimes use $\Phi$ to denote this mapping of subsets. 
For example, if $Z \subset \partial X$ is a realizable triple and
$\Scal_Z$ is the associated $2$-cell, then we will 
let $\Phi(Z)$ denote the triple that determines the 
image 2-cell $\Phi(\Scal_{Z})$. 
\end{nota}

We first define $\beta_U$ in the case that $U$ is a maximal strip. 
Let $x \in \partial U \cap \partial X$.
The isomorphism $\Phi$ determines a graph isomorphism 
of $\lk(U)$ onto $\lk(\Phi(U))$. Since $\lk(U)$ is infinite,
$\lk(\Phi(U))$ is infinite, and hence $\Phi(U)$ is a maximal strip. 
Since $U$ and $\Phi(U)$ are maximal strips, each quadruple
$Z \subset \partial U \cap \partial X$ with $\Scal_Z \neq\emptyset$
is of the form $Z_U(x,y)$ where 
$x$ and $y$ belong to different components of $\partial U$. 
Let $\ell_x$ denote the component containing $x$ and let $\ell_x^{\perp}$
denote the component that does not contain $x$. 

Given $x \in \partial U \cap \partial X$, let $L_x$ be the 
line in $\lk(U)$ consisting of 1-cells $Z_U(x,y)$ with $y \in \ell^\perp_x$. 
The image of $L_x$ is a line in $\lk(\Phi(U))$ and hence there 
exists unique $\beta(x) \in  \Phi(U)\cap \partial X$
such that
\begin{equation}  \label{BetaStrip}
  \Phi(L_{x})~ =~ L_{\beta(x)}.
\end{equation}

Since $\Phi$ is invertible, the map $\beta: \partial U \cap \partial X 
\rightarrow \partial \Phi(U) \cap \partial X'$ is invertible.
Indeed, given $x' \in   \partial \Phi(U) \cap \partial X'$,
define $\beta^{-1}(x')$ by the identity $L_{\beta^{-1}(x')}= \Phi^{-1}(L_{x'})$.

\begin{prop}
Let $U$ be a maximal strip. The map $\beta_U$ defined in (\ref{BetaStrip})
is adapted to $\Phi$. Moreover, 
\begin{equation}  \label{Conjugate}
 \beta_U \circ s_U~ =~ s_{\Phi(U)}\circ \beta_U. 
\end{equation}
\end{prop}

\begin{proof}
In this proof, we will suppress the subscripts `$U$' and `$\Phi(U)$'
since they will be clear from the context. 

Let $y \in \ell_x^{\perp}$.  Then $Z(x,y)$ is a 1-cell that belongs to $L_x$,
and hence  $\Phi(Z(x,y))$ belongs to the line  $L_{\beta(x)}$. In particular, 
there exists $y' \in \ell_{\beta(x)}^{\perp}$ 
such that $\Phi(Z(x,y))= Z(\beta(x), y')$.
Note that $Z(x,y)= Z(y, x)$ and hence $\Phi(Z(y,x))= Z(\beta(y),x)$
for some $x' \in \ell_{\beta(y)}^{\perp}$. 
Hence since $Z(\beta(x), y')= Z( y', \beta(x))$, we have
\begin{equation} \label{TitTat}
   Z(\beta(y),x')~ =~   \Phi(Z(y,x))~  =~ Z( y', \beta(x)).
\end{equation}
Since $L_{x} \cap L_y =Z(x,y)$, we have $L_{\beta(x)} \neq L_{\beta(y)}$.
In particular, $\beta(x) \neq \beta(y)$. Hence it follows from (\ref{TitTat}) 
that 
\begin{equation} \label{BothSides}
  \Phi(Z(x,y))~ =~ Z(\beta(x), \beta(y)).  
\end{equation}

By Proposition \ref{1-CellsFollow}, the $1$-cell $Z(s_U(x), y)$ 
follows $Z(x,y)$ in $\lk(U)$. Thus, since $\Phi$ is orientation preserving,
the $1$-cell  $\Phi(Z(s(x), y))$  follows $\Phi(Z(x,y))$ in $\lk(\Phi(U))$.
Or, equivalently by (\ref{BothSides}), the $1$-cell 
$Z(\beta(s(x)), \beta(y))$  follows $Z(\beta(x),\beta(y))$.
But by  Proposition \ref{1-CellsFollow}, the $1$-cell 
$Z(s(\beta(x)), \beta(y))$  follows $Z(\beta(x),\beta(y))$.
Hence 
\[ Z(s(\beta(x)), \beta(y))~ =~ Z(\beta(s(x)), \beta(y)) \]
and therefore (\ref{Conjugate}) holds as desired. 

If $Z \subset \partial U \cap \partial X$ is realizable,
then since $U$ is a strip, then either $Z =  \partial U \cap \partial X$,
$Z= Z(x,y)$ for $x$, $y$ on different components of $\partial U$, or 
$Z= Z(x,y) \cap Z(x, s(y))$ for  $x$, $y$ on different components. 
Since $\beta$ is a surjection, in the first case we have 
$\Phi(\Scal_{\partial U \cap  \partial X})
   = \Scal_{\beta(\partial U \cap  \partial X)}$.
Suppose that $x,y \in \partial X$ belong to different components
of $\partial U$.  By (\ref{BothSides}) and (\ref{Conjugate}), we have 
\begin{eqnarray*}
  \Phi(\{x, s(x), y, s(y)\}) &=&  \{ \beta(x), s(\beta(x)), \beta(y), s(\beta(y))\} \\
  &=&   \{ \beta(x), \beta(s(x)), \beta(y), \beta(s(y))\}. 
\end{eqnarray*}    
Hence 
\[    \Phi(\{x, s(x), s(y), s^2(y)\})~ =~  
      \{ \beta(x), \beta(s(x)), \beta(s(y)), \beta(s^2(y))\} \]
and thus since $\Phi(Z(x,y) \cap Z(x, s(y)))= \Phi(Z(x,y)) \cap \Phi(Z(x, s(y)))$,
\[    \Phi(\{x, s(x), s(y)\})~ =~  
      \{ \beta(x), \beta(s(x)), \beta(s(y))\}. \]
\end{proof}

\begin{lem} \label{PhiConsecutive}
A quadruple $Z$ is consecutive with respect to 
the rigid conic $U$ if and only if  $\Phi(Z)$ is consecutive
with respect to $\Phi(U)$. 
\end{lem}

\begin{proof}
If $Z$ is consecutive, then exactly two 3-cycles 
contain $Z$. Let $\{Z, Z', Z''\}$ be such a $3$-cycle. 
By Proposition \ref{HexagonDirected}, either  
both $Z'$ and $Z''$ follow $Z$ or $Z$ follows $Z'$ and $Z''$.  
Since $\Phi$ is orientation preserving, 
we have that either both $\Phi(Z')$ and $\Phi(Z'')$ 
follow $\Phi(Z)$ or $\Phi(Z)$ follows $\Phi(Z')$ and $\Phi(Z'')$.
Hence by applying Proposition \ref{HexagonDirected}
to $\{\Phi(Z), \Phi(Z'), \Phi(Z'')\}$ we find that 
$\Phi(Z)$ is consecutive. 

The converse follows by considering $\Phi^{-1}$. 
\end{proof}

We now define $\beta_U$ in the case that $U$ is a 
rigid ellipse interior. 
If $\Card(\partial U \cap \partial X)=n$, then 
the set of consecutive vertices constitute a directed $n$-cycle
in $\lk(U)$. 
Lemma \ref{PhiConsecutive} implies that $\Phi$
maps this directed $n$-cycle bijectively onto the directed $n$-cycle
of consecutive vertices in $\lk(\Phi(U))$. 
For each $x \in \partial U \cap \partial X$, define 
$\beta_U(x)$ be the unique element of 
$\partial \Phi(U) \cap \partial X'$ such that
\begin{equation}  \label{BetaEllipse}
   \Phi \left(Z_U\left(x,s_U^2(x)\right) \right)~ =~ 
Z_{\Phi(U)} \left(\beta(x), s_{\Phi(U)}^2\left(\beta(x)\right)\right). 
\end{equation}

\begin{prop}
Let $U$ be a rigid ellipse interior. The map $\beta_U$ defined by (\ref{BetaEllipse})
is adapted to $\Phi$. Moreover, 
\begin{equation}  \label{ConjugateEllipse}
 \beta_U \circ s_U~ =~ s_{\Phi(U)}\circ \beta_U. 
\end{equation}
\end{prop}

\begin{proof}
Since $Z(s(x), s^3(x))$ follows $Z(x, s^2(x))$ and $\Phi$
is orientation preserving, $Z(\beta(s(x)), \beta(s^3(x)))$ follows
$Z(\beta(x), \beta(s^2(x)))$. But 
$Z(s(\beta(x)), s^3(\beta(x))$ follows $Z(\beta(x), \beta(s^2(x)))$
and so $s(\beta(x))= \beta(s(x))$. 

Each realizable quadruple $Z \subset \partial U \cap \partial X$
is of the form $Z(x,s^{n+1}(x))$ for some $x \in \partial X$ and  $n \geq 1$.  
If $n=1$, then it follows from (\ref{BetaEllipse})
and (\ref{ConjugateEllipse}) that 
\begin{equation}  \label{NonconsecutiveInduction}
\Phi(Z(x,s^{n+1}(x))= \beta(Z(x,s^{n+1}(x))).  
\end{equation}
If (\ref{NonconsecutiveInduction}) holds for some $n$,
then by using (\ref{ConjugateEllipse}), 
the fact that $Z(x,s^{n+2}(x))$ follows $Z(x,s^{n+1}(x))$, the fact
that $Z(\beta(x),s^{n+2}(\beta(x)))$ follows $Z(\beta(x),s^{n+1}(\beta(x)))$, and the fact 
that $\Phi$ is orientation preserving, we find that
(\ref{NonconsecutiveInduction}) holds for $n+1$.
Hence for each quadruple $Z$ we have $\Phi(Z)=\beta(Z)$.  
The claim for triples follows by considering intersections
and the claim for rigid conics is clear. 
\end{proof}

We summarize this section with the following. 

\begin{thm} \label{Bijection}
If $\Phi: \Scal_3(X,\mu) \rightarrow \Scal_3(X', \mu')$ is an
orientation preserving cell complex isomorphism, then the   
exists a unique bijection $\beta: \partial X \rightarrow\partial X'$
such that for each realizable $Z \subset \partial X$ we have
\[ \Phi(\Scal_Z)~ =~ \Scal_{\beta(Z)}. \]
Moreover, for each rigid conic $U \subset X$, we have
\[    \beta \circ s_U~ =~ s_{\Phi(U)} \circ \beta. \]
\end{thm}


\section{Configurations, oriented bisectors,  and eigenlines}  \label{SectionBisectors}

In this section we will prove a geometric lemma
that will be used in the following section to prove Theorem  \ref{Rebuild}.
We will assume throughout that $(X, \mu)$ is the universal covering of a precompact 
translation surface with nonempty and finite frontier. 

\begin{defn}
We will say that two subconics in the plane are {\em equivalent up to homothety} 
if and only if they differ by a homothety and/or a translation. 
Let $[U]$ denote the equivalence class of the subconic $[U]$.
\end{defn}

Two (non-circular) 
ellipse interiors are equivalent if and only if their major axes are parallel and they have 
equal eccentricities.\footnote{The referee has informed us that some refer to this
data as the `complex dilatation'.}

A {\em configuration of subconics about an ellipse} consists 
of an ellipse interior $U$, a finite subset $Z \subset \partial U$
with $\Card(Z)\geq 5$, and for each nonconsecutive pair
$x,y \in Z$, a subconic $U_{x,y} \in \Ecal(\R^2)$ such that 
$Z \cap \partial U_{x,y}=Z(x,y)$. 

\begin{figure}   
\begin{center}

\includegraphics[totalheight=3in]{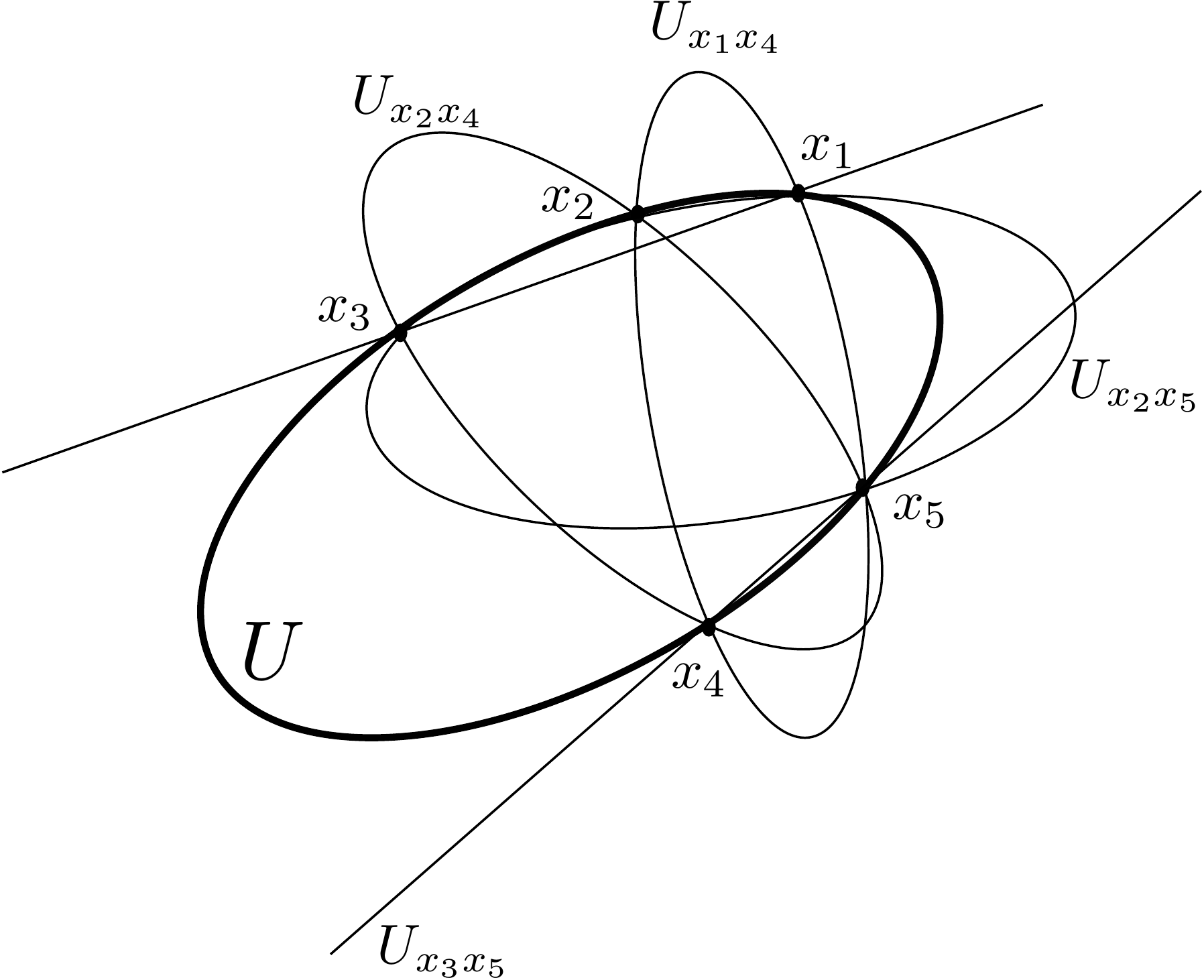}

\end{center}
\caption{A configuration of subconics about the the ellipse $U$.}
\end{figure}

\begin{lem}[Geometric Lemma] \label{MainGeometricLemma}
Let $(U,Z,U_{x,y})$ and $(U',Z',U_{x,y}')$ be two configurations
of subconics about an ellipse. If $[U]=[U']$ and 
there exists a bijection $\beta: Z \rightarrow Z'$ 
such that $\beta \circ s_{U}= s_{U'} \circ \beta$, and 
for each nonconsecutive pair $x,y \in Z$ we have 
\[  \left[U_{x,y} \right]~ =~  \left[U_{\beta(x),\beta(y)}' \right], \]
then for each $x \in Z$, the line $\ell(x,s(x))$ 
is parallel to the line $\ell(\beta(x), \beta(s(x)))$.
\end{lem}

The remainder of this section is dedicated 
to proving Lemma \ref{MainGeometricLemma}.
We first discuss the geometry of eigenvectors 
for degenerate quadratic forms in $\R^2$.
Then at the end of the section, we use this 
discussion to prove Lemma  \ref{MainGeometricLemma}.

Let $q \in Q(\R^2)$ be a quadratic form of signature $(2,0)$.
Let $SO(q)$ denote the group of orientation preserving linear 
transformations that preserve $q$. This group acts transitively 
and freely on the `ellipse' $q^{-1}(1)$, and thus $SO(q)$ is homeomorphic
to a circle. There exists a unique orientation preserving universal 
covering map $e: \R \rightarrow SO(q)$ such that $e$ is a group homomorphism
and $2\pi= \inf\{\theta>0~ |~ e(\theta)={\rm Id}\}$. Of course, 
one should be thinking of the example $q(x_1,x_2)=x_1^2+ x_2^2$ in which 
case 
\[ e(\theta)~ =~ 
\left( \begin{array}{cc}  \cos(\theta) & \sin(\theta) \\
           -\sin(\theta) & \cos(\theta) \end{array} \right).
\]

\begin{lem}  \label{Bisector}
Let $\theta_- \neq \theta_+ \in SO(q)$ and let $x \in \R^2$. The point
\[ e\left( \frac{\theta_- + \theta_+}{2}\right) \cdot x \]
is $q$-orthogonal to  $e\left(\theta_+\right)\cdot x -e\left(\theta_-\right)\cdot x$.
\end{lem}

\begin{proof}
Let $v_{\pm}=e(\theta_{\pm}) \cdot x$. The orthogonal complement of
$v_{+}-v_-$ is the fixed point set of the involution 
$\iota: \R^2 \rightarrow \R^2$ defined by
\[  \iota(w)~ =~ w~ -~ 
 \frac{2 \cdot q(w,v_+-v_-)}{q(v_+-v_-,v_+-v_-)}\cdot (v_+-v_-). 
\] 
We also have the involution $\iota':\R \rightarrow \R$ 
defined by $\iota'(\theta)= \theta - (\theta_+-\theta_-)/2$.
One checks that $e \circ \iota'= \iota \circ e$. 
Since $(\theta_-+ \theta_+)/2$ is a fixed point of 
$\iota'$, the claim  follows.
\end{proof}

Let $(x,y)$ be an ordered pair of distinct vectors belonging to $q^{-1}(r^2)$,
the $q$-circle of radius $r$. Let 
\[ \theta~ =~ \inf \left\{ \theta'>0~ |~ e(\theta')\cdot x =y \right\}, \]
and define
\[  v_{xy}~ =~ e\left(\frac{\theta}{2} \right) \cdot x, \]
Note that $v_{xy}=-v_{yx}$ and hence $\alpha_{vw}= - \alpha_{wv}$.
We will refer to $v_{xy}$ as the {\em oriented $q$-bisector} of $(x,y)$.

\begin{defn}
Let $F=(x_0, \ldots, x_n) \subset q^{-1}(r^2)$ be an ordered $n$-tuple
of distinct points. We will say that the
{\em (ordering of) $F$ is  compatible with $e$} 
if and only if there exist $0 =\theta_0< \cdots < \theta_n \leq 2 \pi$ 
such that $x_i= e(\theta_i) \cdot x_0$.
\end{defn}

For example, let $x, y$ be distinct points in $q^{-1}(r^2)$. The ordered
triple $(x, v_{xy},y)$ is compatible with $e$, but $(x,v_{yx},y)$ is not.

Let $q(\cdot, \cdot)$ denote the polarization of the quadratic form $q$.
For each ordered pair $(x,y)$ of distinct points in $q^{-1}(r^2)$,
define the linear form  $\alpha_{xy}$ that is $q$-dual to 
$v_{xy}$ by
\begin{equation} \label{Alpha}
  \alpha_{xy}(w)~ =~ q(w,v_{xy}). 
\end{equation}
For each ordered quadruple $Q=(x_0,x_1,x_2,x_3)$ of distinct points, 
define a quadratic form $q_Q$ by%
\begin{equation} \label{qQ}
 q_{Q}(v)~ =~ - \alpha_{x_0x_1}(v) \cdot \alpha_{x_2x_3}(v). 
\end{equation}

The signature of $q_{Q}$ is either $(1,0)$, $(0,1)$, or $(1,1)$.
Hence, the form $q_{Q}$ has two distinct eigenvalues, 
$\lambda^-\leq 0 \leq \lambda^+$, with respect to $q$. 
Let $L _{Q}^{\pm}$ denote the 1-dimensional eigenspace
associated to $\lambda^{\pm}$.

\begin{prop}  \label{PositiveEigenvector}
The vector $v_{x_2x_3}-v_{x_0x_1}$ belongs to $L^+_Q$.
\end{prop}

\begin{proof}
The polarization of $q_Q$ is given by
\[ q_Q(v,w)~ =~ \alpha_{x_0x_1}(v)\cdot \alpha_{x_2x_3}(w)~ 
  +~\alpha_{x_0x_1}(w)\cdot \alpha_{x_2x_3}(v). \]
Thus, using (\ref{Alpha}) and the fact that $q(v_{x_0x_1})=r^2=q(v_{x_2x_3})$,
we find that for each $w\in \R^2$
\[ q_Q(v_{x_2x_3}-v_{x_0x_1},w)~ =~ \left(r^2-q(v_{x_2x_3},v_{x_0x_1})\right)
           \cdot q(v_{x_2x_3}-v_{x_0x_1}). \]
Thus, by definition, $v_{x_2x_3}-v_{x_0x_1}$  is an eigenvector
with eigenvalue $r^2-q(v_{x_2x_3}-v_{x_0x_1})$. Since, by assumption, 
$v_{x_2x_3} \neq v_{x_0x_1}$, this eigenvalue is positive.
\end{proof}

\begin{prop}
Let $0 \leq \theta_0<\theta_1<\theta_2<\theta_3 < 2 \pi$ 
and set $x_i=e(\theta_i) \cdot x$. The vector 
\[  u_Q~ :=~ e\left( \frac{\theta_0+\theta_1+\theta_2+ \theta_3}{4} \right) \cdot x\]
belongs to the eigenline $L^-_Q$ where $Q=(x_0,x_1,x_2,x_3)$.
\end{prop}

We will call $u_Q$, the {\em eigenvector associated} to the ordered quadruple $Q$.

\begin{proof}
Since $0<\theta_1-\theta_0< 2 \pi$, we 
have $\theta_1-\theta_0=\inf\{\theta>0~ |~ e(\theta)\cdot x_0=x_1\}$.
Thus, $v_{x_0x_1}= e((\theta_0+\theta_1)/2)\cdot x$. 
Similarly, $v_{x_2x_3} =e((\theta_2+\theta_3)/2) \cdot x$.
Thus, by Lemma \ref{Bisector}, $u_Q$ is orthogonal to $v_{x_2x_3}-v_{x_0x_1}$.
Therefore, since $L^+$ and $L^-$ are orthogonal, the claim follows
from Proposition \ref{PositiveEigenvector}.
\end{proof}

\begin{figure}   
\begin{center}

\includegraphics[totalheight=3in]{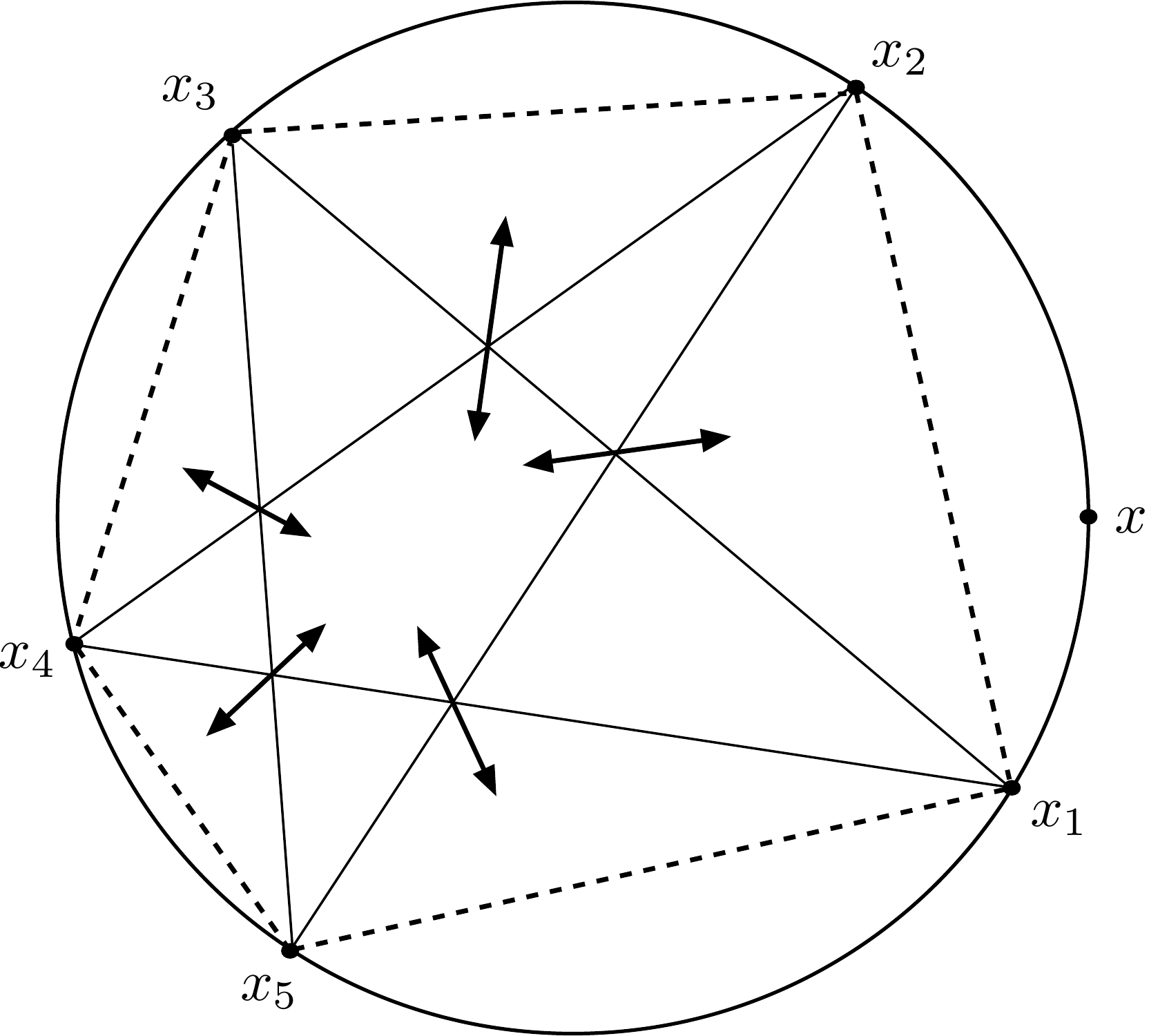}

\end{center}
\caption{The geometric content of Proposition \ref{Intermediate}:
The directions of the bisectors  ($\leftrightarrow$) of successive diagonals 
determine the directions of the dotted lines.  The direction 
of the bisector of $\overline{x_ix_{i+2}}$ and  $\overline{x_{i+1} x_{i+3}}$
is determined by the average of the four angles,
$(\theta_i+\theta_{i+1} +\theta_{i+2}+ \theta_{i+3})/4$. 
The direction of the dotted line $\overline{x_ix_{i+1}}$ is
determined by $(\theta_i + \theta_{i+1})/2$.}

\end{figure}

We now apply the preceding discussion to obtain an intermediate form
of the geometric lemma.  Let $Z$ be a finite subset 
of $q^{-1}(r^2)$ that contains at least five points. 
Let $s: Z \rightarrow Z$ be the successor 
function associated to the counter-clockwise orientation.
In particular, $s(x)=y$ if and only if there exists $\theta>0$
such that $e(\theta \cdot x)=y$ and $e((0, \theta')\cdot x) \cap Z=\emptyset$.

Recall that, given two nonsuccessive points in $Z$, the symbol
$Z(x,y)$ denotes the ordered quadruple  $(x,s(x), y, s(y))$.

\begin{prop} \label{Intermediate}
Let $Z$ and $Z'$ be finite subsets of the unit circle $\partial U_{q_0}$ 
each with cardinality at least five.  If there exists a bijection 
$\beta: Z \rightarrow Z'$ such that $\beta \circ s = s \circ \beta$ 
and for each pair of nonsuccessive vertices $x,y \in Z$, we have 
\[ u_{Z(x,y)}~  =~ u_{Z(\beta(x),\beta(y))},  \]
then for each $x \in Z$, we have
that 
\[ v_{xs(x)}~  =~ v_{\beta(x)s(\beta(x))}.  \]
\end{prop} 
 
\begin{proof}
Let $x \in Z$. For $-2 < i \leq 3$, define 
\[  \theta_i~  =~ \inf \left\{\theta > 0~  |~  
   e(\theta_i) \cdot s^{-2}(x) = s^{i}(x) \right\}. \] 
Since $\Card(Z)\geq 5$, we have 
\[0~ =~ \theta_{-2}<\theta_{-1} < \theta_{0}< \theta_1< \theta_2 \leq \theta_3\leq 2 \pi. \]
Similarly, we obtain
\[ 0~ =~ \theta_{-2}'<\theta_{-1}' < \theta_{0}'< \theta_1'< \theta_2' \leq \theta_3'\leq 2 \pi. \] 
so that $e(\theta_i)=s^i(\beta(x))$.

It follows from the hypothesis that 
\begin{eqnarray*}
  \frac{\theta_{0}+ \theta_{1}+ \theta_{-2} + \theta_{-1}}{4} 
      &=& \frac{\theta_{0}'+ \theta_{1}'+ \theta_{-2}' + \theta_{-1}'}{4}  \\ 
  \frac{\theta_{-2}+ \theta_{-1}+ \theta_{1} + \theta_{2}}{4}
    &=& \frac{\theta_{-2}'+ \theta_{-1}'+ \theta_{1}' + \theta_{2}'}{4}  \\
  \frac{\theta_{0}+ \theta_{1}+ \theta_{2} + \theta_{3}}{4}  
     &=& \frac{\theta_{0}'+ \theta_{1}'+ \theta_{2}' + \theta_{3}'}{4}  \\
  \frac{\theta_{2}+ \theta_{3}+ \theta_{-1} + \theta_{0}}{4}  
   &=&  \frac{\theta_{2}'+ \theta_{3}'+ \theta_{-1}' + \theta_{0}'}{4}  \\
    \frac{\theta_{1}+ \theta_{2}+ \theta_{-1} + \theta_{0}}{4}
   &=&  \frac{\theta_{1}'+ \theta_{2}'+ \theta_{-1}' + \theta_{0}'}{4}
\end{eqnarray*}
By taking the alternating sum of the left hand sides of these equations
and the alternating sum of the right hand sides, we find that
\[ \frac{\theta_0+\theta_1}{2}~ =~  \frac{\theta_0'+\theta_1'}{2}. \]
The claim follows.  
\end{proof}

To  prove the main geometric lemma, we shift the context 
to quadratic forms on $\R^3$. In particular, $q$ will denote an
element of $Q(\R^3)$ and $\underline{q}$  will denote
the `restriction' to the first two coordinates
\[  \underline{q}(x_1,x_2)~ =~ q(x_1,x_2, 0). \]

\begin{prop} \label{Translate}
Let $q \in Q(\R^3)$.  If $q$ has signature $(2,1)$ and $\underline{q}$ has signature $(2,0)$, 
then there exists a (unique) homothety $h$ and a (unique) translation $\tau$ so that 
$h \circ \tau(\partial U_q)= \underline{q}^{-1}(1)$. 
\end{prop}
\begin{proof}
Let $A=(a_{ij})$ be the Gram determinant of $q$ with respect to the standard basis 
for $\R^3$.  We have 
\[  q(x_1,x_2,x_3)~ =~ a_{11}x_1^2~ +~ a_{22}x_2^2~ +~ a_{33}x_3^2~ +~ 2 a_{12}x_1x_2~ +~
     2 a_{13} x_1x_3~ +~ 2a_{23} x_2 x_3. \]
We have 
\[  \underline{q}(x_1,x_2)~ =~  a_{11}x_1^2~ +~ 2 a_{12}x_1x_2~ +~ a_{22}x_2^2. \]
Let $\vec{x}=(x_1,x_2)$ and $\vec{t}=(t_1,t_2)$.
A straightforward calculation shows that 
\begin{equation} \label{qRestrict}
  q(x_1-t_1, x_2-t_2,1)~ =~ \underline{q}(\vec{x})~ 
    +~  2 \left(\vec{a}~ -~ \vec{t}\cdot \underline{A}~  \right) \cdot \vec{x}~ 
  +~ \underline{q}\left(\vec{t} \right)~ -~ 2 \vec{t} \cdot \vec{a}~ +~  a_{33}.  
\end{equation}
where $\vec{a}=(a_{13}, a_{23})$ and  
\[  \underline{A}~ =~ \left( \begin{array}{cc} 
        a_{11} &  a_{12} \\
        a_{12} &  a_{22} 
      \end{array} \right).
\]
Set 
\[  \vec{t}~ =~  \vec{a} \cdot A^{-1}. \]
One computes that 
\[  \underline{q}\left(\vec{t} \right)~ -~ 2 \vec{t} \cdot \vec{a}~ +~  a_{33}~ =~
  \frac{\det(A)}{\det(\underline{A})}, \]
and hence from (\ref{qRestrict}) it 
follows that $\vec{x}-\vec{t} \in \partial U_q$ iff
\begin{equation} \label{dets}
   \underline{q}(\vec{x})~ =~ - \frac{\det(A)}{\det(\underline{A})}. 
\end{equation}
Since $q$ has signature $(2,1)$, we have  $\det(A)<0$, and 
since $\underline{q}$ has signature $(2,0)$ we have that $\det(\underline{A})>0$.
Therefore, the right hand side of (\ref{dets}) is positive. 

Let $\tau(\vec{x})= \vec{x}+ \vec{t}$ and 
$h(\vec{x}) = \left(-\det(A)/\det(\underline{A})\right)^{-\frac{1}{2}} \cdot \vec{x}$.
\end{proof}

\begin{proof}[Proof of Lemma \ref{MainGeometricLemma}]
Note that the set of weights on a configuration is invariant 
under homotheties and translations.\footnote{If a translation 
or homothety is applied to one element of a configuration, then it should 
be applied to all members of the configuration.} 
Also note that parallelism is invariant under homotheties
and translations. In particular, without loss of generality, we have $U=U'$.

Since $U$ is an ellipse interior, there exists a quadratic form 
$q \in Q(\R^3)$ of signature $(2,1)$ such that $U_q=U$ and such that
$\underline{q}$ has signature $(2,0)$. By Proposition \ref{Translate},
there exists a homothety $h$ and a translation $\tau$ 
such that $h \circ \tau(U)= \underline{q}^{-1}(1)$. 
It follows that without loss of generality, $U= \underline{q}^{-1}(1)$.

Let $x,y \in Z$ be distinct, nonsuccessive points.
We will consider three quadratic forms that belong to 
the plane, $Q_{\widehat{Z(x,y)}}$, consisting of quadratic 
forms that vanish on the quadruple $\widehat{Z(x,y)}$. 

The first form is $q$. 

The second form is a quadratic form $r_{xy} $ such that $U_{r_{xy}}=U_{xy}$.
Note that since $\partial U_{xy}\cap Z= Z(x,y)$, we have $r_{xy} \in Q_{\widehat{Z(x,y)}}$.      
Since $\Card(Z) \geq 5$ and $U_{xy} \neq U$, the restriction
of $r_{xy}$ is negative on the arc in $\widehat{\partial U}$ that joins $\hat{x}$ 
to $\widehat{s(x)}$.

The third quadratic form is a degenerate form defined as follows.  
Let $\eta_{xy} \in (\R^3)^*$ be a linear form such that 
$\eta_{xy}(\hat{x})= 0=\eta_{xy}(\widehat{s(x)})$ and $\eta_{xy}$ is negative
on the interior of the convex hull of $Z(x,y)$.
Let $\eta_{xy}' \in (\R^3)^*$ be  such that 
$\eta_{xy}'(\hat{y})= 0=\eta_{xy}'(\widehat{s(y)})$ and $\eta_{xy}'$ is negative
on the interior of the convex hull of $\widehat{Z(x,y)}$. 
The third quadratic form is defined by $p_{xy}=-\eta_{xy} \cdot \eta_{xy}'$.
Note that $p_{x,y} \in  Q_{\widehat{Z(x,y)}}$ and
$p_{xy}$ is positive on the arc in $\widehat{\partial U}$ that 
joins $\hat{x}$ to $\widehat{s(x)}$.

Note that all three forms are nonpositive on the 
convex hull of $\widehat{Z(X, Y)}$. Thus,
since $q$ vanishes on the arc in $\widehat{\partial U}$ that 
joins $\hat{x}$ to $\widehat{s(x)}$ and $\dim(Q_{\widehat{Z(x,y)}})=2$,
there exist positive $a,b \in \R^+$ such that 
\begin{equation} \label{EigenCombination}
   q~ =~ a \cdot r_{x,y}~ + b \cdot p_{x,y}. 
\end{equation}

Since $r_{xy}$ is not a multiple of $q$ (resp. $p_{x,y}$ is
not a multiple of $q$), the form $\underline{r_{x,y}}$ (resp.  $\underline{p_{x,y}}$) 
has two distinct eigenvalues $\mu_-< \mu_+$ (resp. $\nu_+< \nu-$)
with respect to $\underline{q}$. Let 
$M^{\pm}_{x,y}$ (resp. $N^{\pm}_{x,y}$) denote the eigenspace
 of $\underline{r_{x,y}}$ (resp. $\underline{p_{x,y}}$) 
associated to $\mu_{\pm}$ (resp. $\nu_{\mp}$).
It follows from (\ref{EigenCombination}) 
that $M^{\pm}_{x,y}= N^{\mp}_{x,y}$ and  
\[  b \cdot \nu_{\pm}~ =~ 1~ -~ a \cdot \mu_{\mp}. \]

The discussion above applies equally well to a pair of nonadjacent 
distinct points $x', y' \in Z'$. We define forms, $r_{x',y'}$, $p_{x',y'}$,
and eigenspaces, $M^{\pm}_{x',y'} = N^{\mp}_{x',y'}$, in an analogous fashion. 

Since, by hypothesis, $[U_{x,y}]=[U_{\beta(x),\beta(y)}']$,
the forms $\underline{r_{x,y}}$ and $\underline{r_{\beta(x),\beta(y)}}$ 
differ by a positive constant multiple. In particular, 
$M^+_{x,y}=M^+_{\beta(x), \beta(y)}$, and thus $N^-_{x,y}=N^-_{\beta(x), \beta(y)}$.

Let $q_{Z(x,y)} \in Q(\R^2)$ be defined as in (\ref{qQ}). Note that the form  
$\underline{r_{x,y}}$ is a positive multiple of $q_{Z(x,y)}$ 
In particular, the eigenvector $u_{Z(x,y)}$ belongs to $N^-_{x,y}$.  
Similarly, the eigenvector $u_{Z'(\beta(x), \beta(y))}$ 
belongs to $N^-_{\beta(x), \beta(y)}$.
Since $N^-_{x,y}=N^-_{\beta(x), \beta(y)}$ is one dimensional
and $q(u_{Z(x,y)})=1 =  q(u_{Z'(\beta(x), \beta(y))})$, we have 
$u_{Z(x,y)}= \pm u_{Z'(\beta(x), \beta(y))}$.  It then follows 
from the relation $s' \circ \beta= \beta \circ s$ that we have equality. 

Thus, by Proposition \ref{Intermediate}, for each $x \in Z$ we have
$v_{xs(x)}=v_{\beta(x)\beta(s(x))}$. Note that $v_{xs(x)}$ (resp. $v_{\beta(x)\beta(s(x))}$)
is perpendicular to the line  $\ell(s, s(x))$ (resp. $\ell(\beta(x),\beta(s(x)))$).
\end{proof}


\section{Isomorphisms and affine  homeomorphisms} \label{SectionRebuild}

In this section, we prove Theorem \ref{Rebuild}. We begin with
a lemma that reduces the construction of a homeomorphism to a local problem.

\begin{lem}  \label{PsiLemma}
Let $(X, \mu)$ and $(X', \mu')$ be simply connected translation surfaces 
that cover a precompact translation surface with finite and nonempty frontier.
Let $\Phi:\Scal_3(X,\mu)\rightarrow \Scal_3(X',\mu')$ be an 
orientation preserving isomorphism and let $\beta: \partial X \rightarrow \partial X'$
be the compatible bijection. If for each realizable 
quadruple $Z \subset \partial X$, there exists 
$\psi_Z \in {\rm Aff}^+(\R^2)$ such that for each $z \in Z$,
we have 
\[ \psi_Z(\dev_{\mu}(z))= \dev_{\mu'}(\beta(z)),  \]
then there exists $g \in GL^+(\R^2)$ and a homeomorphism 
$\phi: X \rightarrow X'$ such that $g \circ \mu= \mu' \circ \phi$
and $\phi(U)=\Phi(U)$ for each $U \in \Scal_3(X, \mu)$.
\end{lem}

\begin{proof}
Let $Z,Z' \subset \partial X$ be realizable quadruples such that $\Card(Z \cap Z')=3$.
For each $x \in \dev_{\mu}(Z \cap Z')$, we have $\psi_Z(x)=\psi_{Z'}(x)$.
Since each element of ${\rm Aff}^+(\R^2)$ is determined by its values at 
three points, we have $\psi_Z=\psi_{Z'}$. By Proposition \ref{FourConnected},
the graph $\Scal_4(X, \mu)$ is connected, and hence the function $Z \mapsto \psi_Z$ 
is constant. Let $\psi \in {\rm Aff}^+(\R^2)$ denote this constant value.

Let $U \in \Scal_5(X, \mu)$. The map $\psi$ sends 
$\dev_{\mu}(\partial U \cap \partial X)$ onto  
$\dev_{\mu'}(\partial \Phi(U) \cap \partial X')$.
Thus, since $\Card(\partial U \cap \partial X) \geq 5$
and $\Card(\partial \Phi(U) \cap \partial X') \geq 5$,
we have  $\psi(\dev_{\mu}(U))= \dev_{\mu'}(\Phi(U))$.
In particular, we may define a homeomorphism $\phi_U: U \rightarrow \Phi(U)$ by
\[  \phi_U~ =  \dev_{\mu}|_{\Phi(U)}^{-1}\circ \psi \circ \dev_{\mu}(x). \]

Let $Z$ be a realizable quadruple, and let $U_-, U_+ \in \Scal_5(X,\mu)$ be the 
endpoints of $\Scal_Z$. For each $x \in \dev_{\mu'}(\Phi(U_-) \cap \Phi(U_+))$, we have 
\[  \dev_{\mu}|_{\Phi(U_-)}^{-1}(x)~ =~ \dev_{\mu}|_{\Phi(U_+)}^{-1}(x). \]
Therefore, it follows that  $\phi_{U_+}|_{U_+\cap U_-}= \phi_{U_-}|_{U_+\cap U_-}$.
Since $\Scal_4(X, \mu)$ is a connected graph, for any $U,V \in \Scal_5(X, \mu)$,
we have 
\[   \phi_{U}|_{U\cap V}= \phi_{V}|_{U\cap V}.  \]

Define $\phi: X \rightarrow X'$ by setting 
\[  \phi(x)~ =~ \phi_U(x) \]
if $x \in U \in \Scal_4(X, \mu)$. 
By considering the same construction 
for $\Phi^{-1}$ and $\beta^{-1}$, we obtain an inverse for $\phi$.
Let $g$ denote the differential of $\psi$. Then  
$g \cdot \mu= \mu' \circ \phi$.
\end{proof}

As a first application, we have the following.

\begin{prop}
Let $(X, \mu)$ and $(X', \mu')$ be simply connected translation surfaces 
that cover a precompact translation surface with finite and nonempty frontier.
Let $\Phi:\Scal_3(X,\mu)\rightarrow \Scal_3(X',\mu')$ be an 
orientation preserving isomorphism. 
If each rigid conic $U \in  \Scal_5(X, \mu)$ is a strip, 
then there exist $g \in GL^+_2(\R)$ and a homeomorphism 
$\phi: X \rightarrow X'$ such that $g \cdot \mu= \mu' \circ \phi$.  
\end{prop}

\begin{proof}
Let $\beta: \partial X \rightarrow \partial X'$ be the bijection that is compatible with $\Phi$.

Each realizable quadruple $Z \subset \partial X$ is the intersection of the boundaries 
of two strips. In particular, $\dev_{\mu}(Z)$ is the vertex set of a parallelogram. 
Since $\Phi$ is an isomorphism, each rigid subconic $U \in \Scal_5(X, \mu)$ 
is also a strip. Thus, it follows that $\dev_{\mu}(\beta(Z))$ is also the vertex set 
of a parallelogram. Since $\beta$ is orientation preserving, there exists 
a $\psi_Z \in {\rm Aff}^+(\R^2)$ such that $\psi_Z(\dev_{\mu}(z))= \dev_{\mu'}(\beta(z))$.  
The claim now follows from Lemma \ref{PsiLemma}.
\end{proof}

Recall that the developing map $\dev_{\mu}$ is determined
up to post-composition by translations.  In particular,
each subset $A \subset X$ determines a unique class $[A]:= [ \dev_{\mu}(A)]$. 
 
We will say that a pair $\{x, y\} \in \partial X$ is realizable if 
and only if there exists a subconic $U$ such that $\{x,y\}= \partial U \cap \partial X$.  
Note that since $\partial X$ is discrete, realizability of a pair is equivalent to 
the existence of a geodesic segment joining $x$ and $y$. Indeed, given a 
geodesic segment $\sigma$ joining $x$ and $y$ and $\epsilon>0$, 
there exists an ellipse interior $U$
such that $\{x,y\} \subset \partial U$ and the Hausdorff distance between 
$\sigma$ and $\overline{U}$ is less than $\epsilon$.

\begin{prop} \label{ParallelProp}
Let $(X, \mu)$ and $(X', \mu')$ be simply connected translation surfaces 
that cover a precompact translation surface with finite and nonempty frontier.
Let $\Phi:\Scal_3(X,\mu)\rightarrow \Scal_3(X',\mu')$ be an
orientation preserving isomorphism. 
If for each rigid subconic $U \in \Scal_5(X, \mu)$ we have
\begin{equation} \label{UpToHT}
  [\Phi(U)]~ =~  [U], 
\end{equation}
then for each realizable pair $x,y$, we have 
\begin{equation}  \label{ParallelSaddles}
  \ell_{\mu}(x,y)~ =~ \ell_{\mu'}(\beta(x),\beta(y)).
\end{equation}
\end{prop}

\begin{proof}
If $x,y$ realizable, then there exists a subconic $U \in \Scal_5(X,\mu)$ such that
$\{x,y\} \subset \partial U \cap \partial X$ and either $s_U(x)=y$ or 
$s_U(y,x)$. Without loss of generality $y=s_U(x)$.

If $U$ is a strip, then $\Phi(U)$ is a strip. Since 
$[\Phi(U)]=[U]$, each boundary component of $\dev_{\mu}(U)$
is parallel to each boundary component of 
$\dev_{\mu'}(\Phi(U))$.  It follows that
$\ell_{\mu}(x,y)= \ell_{\mu'}(\beta(x), \beta(y))$.  

If $U$ is an ellipse, then $\Phi(U)$ is an ellipse,
and the claim then follows from Lemma \ref{MainGeometricLemma}.
\end{proof}

\begin{thm} \label{MainTheorem}
Let $(X, \mu)$ and $(X', \mu')$ be translations surfaces that each cover
precompact translation surfaces with finite and nonempty frontier.  
Let $\Phi:\Scal_3(X,\mu)\rightarrow \Scal_3(X',\mu')$ be an isomorphism. 
If for each rigid subconic $U \in \Scal_5(X, \mu)$ we have
\begin{equation} \label{UpToHT2}
  [\Phi(U)]~ =~  [U], 
\end{equation}
then there exists $h \in H$  and a homeomorphism 
$\phi: X \rightarrow X'$ such that $h \cdot \mu= \mu' \circ \phi$
and $\phi(U)=\Phi(U)$ for each $U \in \Scal_3(X, \mu)$.  
\end{thm}

\begin{proof}
We first suppose that $X$ and $X'$ are simply connected. 
Let $\beta: \partial X \rightarrow \partial X'$ be the bijection compatible with $\Phi$.

Let $Z \subset \partial X$ be a realizable quadruple. 
Let $U \in \Scal_Z$ and let $z \in Z =\partial U \cap \partial X$.
By Proposition \ref{ParallelProp}, for each $z, z' \in Z$, we have 
$\ell(\beta(z),\beta(z'))=\ell(z,z')$.  

Let $\{z_1,z_2,z_3\} \subset Z$. Since $Z$ is in general position, 
there exists a unique $\psi_{z_1,z_2,z_3} \in H(\R^2) \cdot T(\R^2)$
such that $\psi(\dev_{\mu}(z_i))=\dev_{\mu'}(\beta(z_i))$ for 
each $i=1,2,3$. Since each element of $H(\R^2) \cdot T(\R^2)$ is
determined by its values on two points, $\psi_{z_1,z_2,z_3}$ does 
not depend on the choice of triple $\{z_1,z_2,z_3\} \subset Z$.
Let $\psi_Z$ denote the common value of $\psi_{z_1,z_2,z_3}$.
The case of simply connected surfaces then follows from Lemma \ref{PsiLemma}.

If $X$ and $X'$ are not simply connected, then we can reduce 
to the simply connected case by considering the universal coverings
$p: \tilde{X} \rightarrow X$ and $p': \tilde{X}' \rightarrow X'$.

The subspaces of ellipses, $\Ecal_3(X,\mu)$ and $\Ecal_3(X', \mu')$, 
are obtained by removing strip vertices from $\Scal_{3}(X, \mu)$ 
and  $\Scal_{3}(X', \mu')$ respectively. Since the link 
of a strip is not homeomorphic to link of any ellipse, 
the homeomorphism $\Phi$ restricts to a homeomorphism from 
$\Ecal_{3}(X, \mu)$ onto $\Ecal_{3}(X', \mu')$.

By Corollary \ref{Covering}, the restriction of $p$ (resp.  $p'$)
to $\Ecal_3(\tilde{X}, \tilde{\mu})$ (resp. $\Ecal_3(\tilde{X}', \tilde{\mu}')$)
is a covering onto $\Ecal_3(X,\mu)$ (resp. $\Ecal_3(X', \mu')$).
It follows that the homeomorphism $\Phi$ on ellispes lifts to a homeomorphism  
$\tilde{\Phi}: \Ecal_3(\tilde{X}, \tilde{\mu}) \rightarrow  \Ecal_3(\tilde{X}', \tilde{\mu}')$.
Moreover, the map $\gamma \rightarrow \Phi \circ \gamma \circ \Phi^{-1}$ defines 
defines an isomorphism from $\Gal(\tilde{X}/X)$ to $\Gal(\tilde{X}/X)$.

We claim that $\tilde{\Phi}$ extends to a homeomorphism from 
$\Scal_3(\tilde{X}, \tilde{\mu})$ to $\Scal_3(\tilde{X}', \tilde{\mu}')$.
Indeed, a given strip vertex $U$ is an extreme point of a convex polygonal
2-cell $Z$ in $\Scal_3(\tilde{X}, \tilde{\mu})$. The map $\tilde{\Phi}$ maps 
$Z$ homeomorphically onto a 2-cell $Z'$ in  $\Scal_3(\tilde{X}', \tilde{\mu})'$.
Moreover each $1$-cell in the boundary of $Z$ is mapped homeomorphically onto 
the corresponding $1$-cell in the boundary of $Z$. 
In sum, we have a homeomorphism between two closed convex polygons 
each with finitely many vertices removed that maps each edge of one 
polygon homeomorphically onto an edge of the other. An elementary argument
provides an extension to a homeomorphism from $\overline{Z}$ to $\overline{Z'}$. 
Since the link of the vertex $U$ is connected, an inductive argument 
shows that the extension to $U$ does not depend on the choice of $Z$.
An inverse for the extension can be constructed by extending $\tilde{\Phi}^{-1}$.

If (\ref{UpToHT2}) holds for each subconic $U$, then it holds for each lifted subconic
$\tilde{U} \subset \tilde{X}$. 
Hence Theorem \ref{MainTheorem} provides a homeomorphism 
$\tilde{\phi}: \tilde{X} \rightarrow \tilde{X}'$ and a homothety $h$ 
such that $h \cdot \mu= \mu' \circ \tilde{\phi}$. Moreover 
$\tilde{\phi}(U)=\Phi(U)$ for each $U \in \Scal_3(\tilde{X}, \tilde{\mu})$, 
and hence the map 
$\gamma \mapsto \tilde{\phi} \circ \gamma \circ \tilde{\phi}^{-1}$ 
defines a isomorphism of $H$ onto $H'$. In particular,
$\tilde{\phi}$ descends to a homeomorphism $\phi: X \rightarrow X$
so that $h \cdot \mu= \mu' \circ \phi$.  
\end{proof}


\section{Characterizing the Veech group}   \label{SectionVeech}

Let $SL(\R^2)$ denote the multiplicative group of $2 \times 2$
matrices with unit determinant.
Recall that the Veech group $\Gamma(X, \mu)$ consists 
of those $g \in SL(\R^2)$ such that
there exists a  homeomorphism $\phi: X \rightarrow X$
with $g \circ \mu = \mu \circ \phi$  \cite{Veech89} \cite{GtkJdg}. 

\begin{thm}
$\Gamma(X, \mu)$ consists 
of those $g \in SL(\R^2)$ such that
there exists an orientation preserving 
homeomorphism $\Phi: \Scal_3(X, \mu) \rightarrow \Scal_3(X, \mu)$
such that for each $U \in \Scal_5(X, \mu)$ we have 
$[\Phi(U)]=[g'(U)]$.
\end{thm}

\begin{proof}
Let $g \in   SL(\R^2)$ and let $\phi: X \rightarrow X$
be a homeomorphism such that $g \circ \mu = \mu \circ \phi$.  
The homeomorphism $\phi$ acts on subsets of $X$. 
Since $g \circ \mu = \mu \circ \phi$, we have
\begin{equation} \label{phiBrackets}
 [\dev_{\mu}(\phi(U))]~  =~ [g \cdot\dev_{\mu}(U)],  
\end{equation}
In particular, if $U$ is a subconic then 
the set $\dev_{\mu}(\phi(U))$ is also a subconic.
Since   $g \circ \mu = \mu \circ \phi$, the homeomorphism 
$\phi$ extends to homeomorphism 
$\overline{\phi}: \overline{X} \rightarrow \overline{X}$.
If $\Card(\partial U \cap \partial X)=n$, then we have
$\Card(\overline{\phi}(\partial U \cap \partial X))=n$.
Thus $\phi$ defines a homeomorphism 
from $\Scal_3(X, \mu) \rightarrow \Scal_3(X, \mu)$
that satisfies (\ref{phiBrackets}).

Conversely, let $\Phi: \Scal_{3}(X, \mu) \rightarrow \Scal_3(X, \mu)$ 
and $g' \in SL(\R^2)$ such that for each $U \in \Scal_5(X, \mu)$ we have 
$[\dev_{\mu}(\Phi(U))]=[g(\dev_{\mu}(U))]$. If we let $\mu'= g \circ \mu$, then 
$\Phi$ defines a homeomorphism from $\Scal_{3}(X, \mu) \rightarrow \Scal_3(X, \mu')$  
such that $[\dev_{\mu}(\Phi(U))]=[\dev_{\mu'}(U)]$.  Hence we can apply 
Theorem \ref{MainTheorem} to obtain a homeomorphism 
$\phi: X \rightarrow X$ and a homothety $h$ such that $h \circ \mu'= \mu \circ \phi$.
In particular, $h \circ g \circ \mu= \mu \circ \phi$. Since $\det(g)=1$,
$\phi$ is a homeomorphism and $(X,\mu)$ has finite area, the homothety
$h$ is the identity. 
\end{proof}

\end{document}